\documentclass{siamart1116}
\usepackage{amsmath,amsfonts,amssymb}
\usepackage{stmaryrd}
\usepackage{latexsym}
\usepackage{epsfig,overpic}
\usepackage{subfigure}
\usepackage{color}
\let\ep\varepsilon
\newcommand{\bA}{\mathbf A}
\newcommand{\bB}{\mathbf B}
\newcommand{\bC}{\mathbf C}
\newcommand{\bI}{\mathbf I}
\newcommand{\bP}{\mathbf P}

\newcommand{\bl}[1]{\mathbf l^{#1}}
\newcommand{\blk}[2]{\mathbf l^{#1,#2}}

\newcommand{\bn}{\mathbf n}
\newcommand{\bp}{\mathbf p}
\newcommand{\bs}{\mathbf s}
\newcommand{\bw}{\mathbf w}
\newcommand{\bx}{\mathbf x}
\newcommand{\by}{\mathbf y}

\newcommand{\bbf}{\mathbf f}

\newcommand{\eoc}[1]{$\text{eoc}_{\texttt{#1}}$}
\newcommand{\eoca}[1]{$\text{eoc}_{\texttt{#1}}^\ast$}

\newcommand{\wn}{{w_N}}
\newcommand{\wne}{{w_N^e}}
\newcommand{\wt}{{\bw_T}}
\newcommand{\wte}{{\bw_T^e}}
\newcommand{\G}[2]{\Gamma^{#1}_{#2}}
\newcommand{\norm}[1]{\Vert #1 \Vert}

\newcommand{\innerprod}[2]{\left( #1, #2 \right)}
\newcommand{\stab}{\rho_n}
\newcommand{\stabold}{\rho_{n-1}}
\newcommand{\consist}{\mathcal{E}_C^n}
\newcommand{\interpol}{\mathcal{E}_I^n}
\newcommand{\err}{\mathbb{E}}

\newcommand{\T}{\mathcal T}
\newcommand{\Div}{\operatorname{\rm div}}

\newcommand{\rr}{\mathbb{R}}
\newcommand{\cc}{\mathbb{C}}
\newcommand{\Gs}{\mathcal{G}} 

\DeclareGraphicsExtensions{.pdf,.eps,.ps,.eps.gz,.ps.gz,.eps.Y}

\newif\ifarxiv
\arxivtrue

\usepackage{todonotes}

\usepackage{booktabs}

\usepackage{siunitx}
\newcommand{\numbf}[1]{{\underline{\num{#1}}}}
\newcommand{\numit}[1]{\underline{\underline{\num{#1}}}}
\newcommand{\numQ}[1]{\num[round-precision=2,round-mode=places, scientific-notation=false]{#1}}
\renewcommand{\O}{\mathcal{O}}

\newcounter{ass}

\newcommand{\asslabel}{}
\newcommand{\ifasslabel}[1]{}

\newtheorem{remark}{Remark}[section]

\begin{document}
\title{A stabilized trace finite element method for partial differential equations on evolving surfaces\thanks{ 
\ifarxiv%
To appear in SINUM. Appendix part of this report is only in arXiv version.
\else%
Revision submitted to the editors on: 21 March 2018.
\fi%
\funding{C.L. was partially supported by the German Science Foundation (DFG) within the project ``LE 3726/1-1''; M.O. was partially supported by NSF through the Division of Mathematical Sciences grants 1717516 and 1522191. X.X. was partially supported by NSFC projects 11571354 and 91630208.}}}
\author{Christoph Lehrenfeld\thanks{Institute for Numerical and Applied Mathematics, University of G\"ottingen, G\"ottingen, Germany, \email{lehrenfeld@math.uni-goettingen.de}, \url{http://num.math.uni-goettingen.de/\string~lehrenfeld/} } \and Maxim A. Olshanskii\thanks{Department of Mathematics, University of Houston, Houston, Texas 77204-3008, and Sechenov University, Moscow 119991, Russian Federation,
\email{molshan@math.uh.edu}, \url{http://www.math.uh.edu/\string~molshan/}}
\and Xianmin Xu\thanks{LSEC, Institute of Computational Mathematics and Scientific/Engineering Computing,
  NCMIS, AMSS, Chinese Academy of Sciences, Beijing 100190, China, \email{xmxu@lsec.cc.ac.cn}, \url{http://lsec.cc.ac.cn/\string~xmxu/}.}}
\maketitle
\begin{abstract}
In this paper, we study a new numerical method for the solution of partial differential
equations on evolving surfaces.   The numerical method is built on the stabilized trace finite element method (TraceFEM) for the spatial discretization and finite differences for the time discretization. The TraceFEM uses a stationary background mesh, which can be chosen independent of time and the position of the surface. The stabilization ensures well-conditioning of the algebraic systems and defines a regular extension of the solution from the surface to its volumetric neighborhood. Having such an extension is essential for the numerical method to be well-defined.
The paper proves numerical stability and optimal order error estimates for the case of simplicial  background meshes and finite element spaces of order $m\ge1$. For the algebraic condition numbers of the resulting systems we prove estimates, which are independent of the position of the interface. The method allows that the surface and its evolution are given implicitly with the help of an indicator function. Results of numerical experiments for a set of 2D evolving surfaces are provided.
\end{abstract}



\section{Introduction}
Partial differential equations on evolving surfaces arise in a number of  mathematical models in natural sciences and engineering.   Well-known examples include the diffusion and transport of  surfactants along interfaces in multiphase fluids \cite{GReusken2011,Milliken,Stone}, diffusion-induced grain boundary motion \cite{GrainBnd1,GrainBnd2} and lipid interactions in moving cell membranes \cite{ElliotStinner,Novaketal}. Thus, recently there has been a significant  interest in developing and analyzing numerical methods for PDEs on time-dependent surfaces; see, for example, the review articles \cite{DEreview,olshanskii2016trace}. The present paper contributes to the field with an unfitted finite element methods for PDEs posed on implicitly defined time-dependent  surfaces and its complete stability and error analysis.

Geometrically unfitted finite element methods  exploit the idea of using a time-independent background finite element space to approximate the solution of a PDE posed on an embedded surface.
The background finite element space is defined on an ambient triangulation, which is not fitted to the surface.
There are several approaches that fit this framework. In the PDE extension approach, one extends the PDE off the surface to a volumetric computational domain in a special way such that the solution of the ambient PDE restricted to the surfaces solves the original problem. Further one solves this new PDE by a conventional discretization method in $\mathbb{R}^3$;
see \cite{bertalmio2001variational} and  \cite{XuZh} for the extension to evolving surface case.  In the trace finite element method, one takes an opposite approach. Instead of extending the surface PDE, one takes the traces of the background volumetric finite element functions on the embedded surface for the purpose of PDE approximation~ \cite{OlshReusken08}. In the TraceFEM, one also may add stabilization terms  which involve the restrictions of the background functions to the tetrahedra cut by the surface~\cite{Alg1}. Several authors have treated PDEs on time-dependent surfaces using this framework. Thus, a  method based on a characteristic-Galerkin formulation combined with the TraceFEM in space was proposed and analysed in \cite{hansbo2015characteristic}. An interesting variant of TraceFEM and narrow-band FEM for a conservation law on an evolving surface was devised in \cite{deckelnick2014unfitted}. A mathematically sound approach which entails rigorous stability and error analysis was investigated in \cite{olshanskii2014eulerian,olshanskii2014error}. In those papers, a PDE on an evolving closed surface $\Gamma(t)\subset\mathbb{R}^3$ was studied as an equation posed on a fixed space--time manifold $\Gs =\bigcup_{t \in (0,T)} \Gamma(t) \times \{t\}\subset \Bbb{R}^{4}$. Further a space--time trace finite element method was applied to approximate the PDE posed on $\Gs$. While the space--time TraceFEM was shown to be provably accurate, its implementation requires the numerical integration over the time slices of  $\Gs$. An algorithm for piecewise tetrahedral reconstruction of  $\Gs$ from the zero of a level-set function can be found in \cite{grande2014eulerian,L_SISC_2015}, but these reconstruction methods are not a part of standard scientific computing software.
Moreover, it remains a challenging problem to build a higher order reconstruction of $\Gs$. Recent attempts to build  geometrically unfitted finite element method that avoids numerical recovery of $\Gs$  are reported in \cite{hansbo2016cut,ERW16,olshanskii2017trace}.
At the time of writing this paper, the authors are unaware of stability or error analysis of these most recent methods that avoid reconstruction of $\Gs$.
Therefore, building an accurate, efficient and reliable unfitted finite element method for PDEs posed on surfaces is still a challenge. In particular, one may want the method to benefit from higher order elements, to avoid a reconstruction of $\Gs$, and to admit rigorous analysis.

The present paper addresses the challenge by suggesting a hybrid finite difference (FD) in time\,/\,TraceFEM in space method, which uses the restrictions of surface independent background FE functions on a steady discrete surface $\G{}{h}(t_n)$ for each time node $t_n$. A standard FD approximation is applied to treat the time derivative. Hence, opposite to the approaches in \cite{grande2014eulerian,olshanskii2014eulerian,olshanskii2014error,L_SISC_2015,hansbo2016cut} a reconstruction of $\Gs$ or numerical integration over $\Gs$ is not needed. Instead one needs an extension of the TraceFEM solution (but not of the PDE!) from $\G{}{h}(t_n)$ to a narrow band of tetrahedra containing $\G{}{h}(t_n)$. In \cite{olshanskii2017trace} it was suggested that a quasi-normal extension of the discrete solution by a variant of the fast marching method (FMM) can be used allowing the modular application of the standard tools: steady-surface TraceFEM and FMM. Numerical experiments in \cite{olshanskii2017trace} demonstrated that the piecewise linear TraceFEM combined with BDF2 in time and a variant of the FMM from~\cite{GReusken2011} is  second-order accurate for $h=\Delta t$, unconditionally stable and capable to handle the case of surfaces undergoing topological changes. Here we build on the approach from \cite{olshanskii2017trace}, with the following important modification: The finite element formulation is augmented with a volumetric integral that includes derivatives of test and trial functions along the quasi-normal directions to   $\G{}{h}(t_n)$. The integral is computed over  tetrahedra cut by the surface at the given time $t_n$ and possibly (depending on the surface normal velocity and the time step size) over a few more layers of the tetrahedra. The benefit of the augmentation is threefold: first, it implicitly defines an extension of the solution to a narrow band of the surface hence eliminating the need for FMM or any other additional modulus; second, it stabilizes the method algebraically leading to well-conditioned matrices; finally, it leads to a concise variational formulation of the method and so allows numerical stability and error analysis. The paper presents such analysis as well  as the analysis of algebraic stability for the fully discrete method (no simplified assumptions are made such as numerical integration over exact surface).
The analysis allows background finite element spaces of arbitrary order $m\ge1$.
We notice however that for $m > 1$ and optimal order convergence, numerical integration with higher order accuracy is required which is a non-trivial task; cf. remark \ref{rem:integration} below.
For the time discretization we apply the backward Euler method. Higher order in time discretizations are straightforward, and we illustrated them in numerical example section, but they are not covered by the presented analysis.

The remainder of the paper is organized as follows. In section~\ref{s:form} we review the surface transport--diffusion equation as an example of a PDE posed on an evolving surface.  To elucidate the main ideas behind the method and analysis, section~\ref{s:time} introduces a semi-discrete method (FD in time \,/\, continuous in space) and presents its stability analysis. Further, in section~\ref{s:FEM} we devise a fully discrete method. In section~\ref{s:Analysis} the core stability and error analysis of the paper is given.
In section~\ref{s:algebra} we prove bounds on condition numbers of resulting matrices, which are independent of how the surface cuts through the background mesh.
Results of several numerical experiments, which illustrate the theoretical findings and show optimal convergence order also in weaker norms, are collected in section~\ref{s:Numerics}. Section~\ref{s:Conclusions} gives some conclusions and discusses interesting open problems.

\section{Mathematical problem} \label{s:form}

Consider a surface $\Gamma(t)$ passively advected by a smooth velocity field $\bw=\bw(\bx,t)$, i.e. the normal velocity of $\Gamma(t)$ is given by $\bw \cdot \bn$, with
$\bn$ the unit normal on $\Gamma(t)$. We assume that for all $t \in [0,T] $,  $\Gamma(t)$ is a smooth hypersurface that is  closed ($\partial \Gamma =\emptyset$), connected, oriented, and contained in a fixed domain $\Omega \subset \Bbb{R}^d$, $d=2,3$. In the remainder we consider $d=3$, but all results have analogs for the case $d=2$.

As an example of the surface  PDE, consider the transport--diffusion equation modelling the conservation of a scalar quantity
$u$ with a diffusive flux on $\Gamma(t)$ (cf. \cite{James04}):
\begin{equation}
\dot{u} + ({\Div}_\Gamma\bw)u -{ \nu}\Delta_{\Gamma} u=0\quad\text{on}~~\Gamma(t), ~~t\in (0,T],
\label{transport}
\end{equation}
 with initial condition $u(\bx,0)=u_0(\bx)$ for $\bx \in \G{0}{} :=\Gamma(0)$.
 Here $\dot{u}$ denotes the advective material derivative, ${\Div}_\Gamma:=\operatorname{tr}\left( (I-\bn\bn^T)\nabla\right)$ is the surface divergence,  $\Delta_\Gamma$ is the  Laplace--Beltrami operator, and $\nu>0$ is the constant diffusion coefficient. The well-posedness of suitable weak formulations of \eqref{transport}  has been proven in \cite{Dziuk07,olshanskii2014eulerian,alphonse2014abstract}.

The equation \eqref{transport} can be written in several equivalent forms, see \cite{DEreview}.  In particular,
for any smooth extension of   $u$ from the space--time manifold
\[
\Gs = \bigcup\limits_{t \in (0,T)} \Gamma(t) \times \{t\},\quad  \Gs\subset \Bbb{R}^{4},
\]
 to a neighborhood of $\Gs$, one can expand $\dot{u}$ using the Cartesian derivatives
\[\dot{u}= \frac{\partial u}{\partial t} + \bw \cdot \nabla u.\]

In this paper, we assume that $\Gamma(t)$ is the zero-level set of a smooth level-set function $\phi(\bx,t)$,
 \[
\Gamma(t)=\{\bx\in\mathbb{R}^3\,:\,\phi(\bx,t)=0\},
\]
such that $|\nabla \phi|\ge c>0$ in $\O(\Gs)$, a neighborhood of $\Gs$.
Note that we do not assume that $\phi$ is a signed distance function. The method that we introduce can deal with more general level set functions. However, the analysis (sections \ref{s:stab:semi-disc}, \ref{s:FEM} and \ref{s:Analysis}) uses the assumption of a level set function with the signed distance property in order to keep the amount of technical details at a comprehensive level. 

For a smooth $u$ defined on $\Gs$, a function $u^e$ denotes the extension of $u$ to $\O(\Gs)$ along spatial normal directions to the level-sets of $\phi$, it holds
$\nabla u^e\cdot\nabla \phi=0$ in $\O(\Gs)$, $u^e=u$ on $\Gs$. The extension $u^e$ is smooth once $\phi$  and $u$ are both smooth.
 Further, we use the same notation $u$ for the function on $\Gs$ and its extension to  $\O(\Gs)$. We shall write $\O(\Gamma(t))$ to denote a neighborhood of $\Gamma(t)$ in $\mathbb{R}^3$, which is the time cross-section of $\Gs$ ,
$\O(\Gamma(t)):=\{\bx\in\mathbb{R}^3\,:\,(\bx,t)\in\O(\Gs)\}$.

We can rewrite  \eqref{transport} as follows:
\begin{equation}
\left\{
\begin{array}{rl}
\frac{\partial u }{\partial t}+\bw\cdot\nabla u + ({\Div}_\Gamma\bw)u -{ \nu}\Delta_{\Gamma} u=0&\text{on}~~\Gamma(t),\\
\nabla u\cdot\nabla\phi=0 &\text{in}~~\O(\Gamma(t))
\end{array}~~t\in (0,T].
\right.
\label{transport1}
\end{equation}
This formulation will be used for the discretization method.

\section{Discretization in time} \label{s:time}
\subsection{Preliminaries and notation}
We introduce notation for the surfaces at discrete time levels.
For simplicity of notation, consider the uniform time step $\Delta t=T/N$, and let  $t_n=n\Delta t$ and $I_n=[t_{n-1},t_n)$.
Denote by $u^n$ an approximation of $u(t_n)$, define $\G{n}{} := \Gamma(t_n)$ and $\phi^n(\bx) := \phi(\bx,t_n)$, $n=0,\dots,N$.
We assume that $\O(\Gs)$ is a sufficiently large neighborhood of $\Gs$ such that
\begin{equation}\label{ass1}
  \G{n}{}\subset\O(\G{n-1}{})\quad\text{for}~n=1,\dots,N, ~ \text{cf. Fig. \ref{fig:Gammaneighborhood}.}
\end{equation}
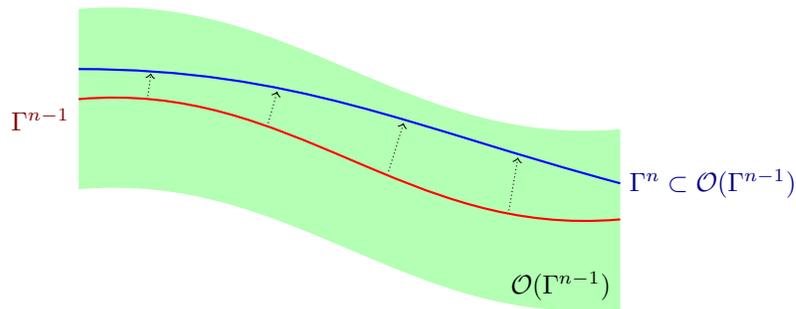
\begin{figure}
  \vspace*{-0.4cm}
  \begin{center}
    \begin{tikzpicture}[scale=0.8]
      \draw[draw=none, fill=green, opacity=0.3] (1,5)
      to[out=5,in=185] (10,3)
      -- (10,0)
      to[out=185,in=5] (1,2)
      -- (1,5)
      ;
      \draw[red,thick] (1,3.5) to[out=5,in=185] (10,1.5);
      \draw[blue,thick] (1,4) to[out=0,in=165] (10,2.1);
      \node[above left] at (10,0) {\color{black} $\O(\G{n-1}{})$};
      \node[below left] at (1,3.5) {\color{red!50!black} $\G{n-1}{}$};
      \node[right] at (10,2.1) {\color{blue!50!black} $\G{n}{} \subset \O(\G{n-1}{})$};
      \draw[->,densely dotted,black] (2.15,3.55) -- (2.2,3.925);
      \draw[->,densely dotted,black] (4.15,3.1) -- (4.3,3.625);
      \draw[->,densely dotted,black] (6.15,2.3) -- (6.4,3.1);
      \draw[->,densely dotted,black] (8.15,1.65) -- (8.3,2.525);
    \end{tikzpicture}
  \end{center}
  \vspace*{-0.4cm}
  \caption{Sketch of interface positions at different time instances and the neighborhood of one of these interfaces.}
  \label{fig:Gammaneighborhood}
\end{figure}
In this case, $u^{n-1}$ is well-defined on $\G{n}{}$.
Further, we use the following abbreviations in norms and scalar products for functions $u,v$ in a domain $G$:
$\innerprod{u}{v}_{G} := \innerprod{u}{v}_{L^2(G)}$,
$\norm{u}_{G} := \norm{u}_{L^2(G)}$, $\norm{u}_{\infty,G} := \norm{u}_{L^{\infty}(G)}$. For a function $v$ defined on $\Gamma(t)$ or on $\O(\Gamma(t))$ we use
\[
\norm{v}_{\infty,I_n}
       := \sup\limits_{t\in I_n} \norm{v}_{\infty,\Gamma(t)},
    \qquad \norm{v}_{\infty}
       := \sup\limits_{t\in [0,T]} \norm{v}_{\infty,\Gamma(t)}.
\]

We also introduce the decomposition $\bw=\wt+\wn\bn$
where $\wt$ and $\wn\bn$ denote the tangential and normal parts of the velocity vector field $\bw$ on $\Gamma(t)$.

\subsection{Time discretization method}
The implicit Euler method for
\eqref{transport1} is
\begin{equation}
\left\{
\begin{split}
\frac{u^n-u^{n-1}}{\Delta t}+\bw^n\cdot\nabla u^n+(\Div_{\Gamma} \bw^n) u^n-\nu\Delta_{\Gamma} u^n&=0\quad\text{on}~~\G{n}{},\\
\nabla u^{n}\cdot\nabla\phi^n&=0\quad\text{in}~~\O(\G{n}{}).
\end{split}\right.
\label{e:ImEuler}
\end{equation}
Obvious modifications are required to devise higher order time discretizations.
For example, for the $O(\Delta t^2)$ method one can use BDF2 approximation of the time derivative, replacing
$\frac{u^n-u^{n-1}}{\Delta t}$ by $\frac{3u^n-4u^{n-1}+u^{n-2}}{2\Delta t} $  in \eqref{e:ImEuler} and additionally assuming $\G{n}{}\subset\O(\G{n-2}{})$, cf. also Remark \ref{rem:bdf2}.\\

\paragraph{Variational formulation in space}
The basis for the spatial discretization is a variational formulation in space. For every time instance $t$ we denote by $\mathcal{V}(t)$ the Hilbert space of functions which are defined in a neighborhood of $\Gamma(t)$ and are constant in the direction of the gradient of $\phi$ (the normal direction),
$\mathcal{V}(t) := \overline{ \mathcal{V}_\ast(t)}^{\Vert \cdot \Vert_{\mathcal{V}}}$  with
\begin{equation} \label{eq:defV}
  \mathcal{V}_\ast(t) := \{ v \in C^2(\O({\Gamma(t)})) \mid \nabla v \cdot \nabla \phi = 0 \} \text{ and } \Vert v \Vert_{\mathcal{V}} := \left(\Vert v \Vert_{H^1(\Gamma(t))}^2 + \Vert \nabla \phi \cdot \nabla v \Vert_{L^2(\O(\Gamma(t)))}^2\right)^{\frac12},
\end{equation}
where $\Vert v \Vert_{H^1(\Gamma(t))}^2 = \Vert v \Vert_{\Gamma(t)}^2 + \Vert \nabla v \Vert_{\Gamma(t)}^2$.
Functions in $\mathcal{V}_\ast(t)$ have weak derivatives in $\mathcal{O}(\Gamma(t))$; cf.
\ifarxiv%
Lemma \ref{lem:app:wd} in the appendix.
\else%
\cite[Lemma 18]{arXiv}.
\fi%
Further, note that on $\mathcal{V}(t)$ there holds $\Vert \cdot \Vert_{H^1(\Gamma(t))} = \Vert \cdot \Vert_{\mathcal{V}}$ and thus $\Vert \cdot \Vert_{H^1(\Gamma(t))}$ is a norm.
Assume $u^{n-1} \in L^2(\G{n-1}{})$ is given,  with $u^{n-1} \in \mathcal{V}(t_{n-1})$ and \eqref{ass1}.
We seek for $u^n \in \mathcal{V}(t_n)$
such that for all $v \in \mathcal{V}(t_n)$ there holds
\begin{equation} \label{e:conttraceFEM1}
  \int_{\G{n}{}}\left(\frac{1}{\Delta t} u^n +\bw\cdot\nabla u^n +(\Div_{\Gamma} \bw) u^n\right) v\, ds+\nu\int_{\G{n}{}} \! \nabla_{\Gamma}u^n\! \cdot\!\nabla_{\Gamma}v \, ds =
  \int_{\G{n}{}} \frac{1}{\Delta t} u^{n-1} v \, ds.
\end{equation}
Note that the second equation of \eqref{e:ImEuler} is hidden in the definition of space $\mathcal{V}(t _n)$ in \eqref{e:conttraceFEM1}. Below, in the finite element method we will impose it weakly through the variational formulation.

\paragraph{Integration by parts characterization of the convection term}
Since $\Gamma(t)$ is smooth and closed, we have the integration by parts identity:
\begin{equation} \label{e:intpartc}
\begin{aligned}
& \int_{\Gamma(t)}(\bw\cdot\nabla u)v\,ds =\int_{\Gamma(t)}(\wt\cdot\nabla_\Gamma u)v\,ds=
-\int_{\Gamma(t)}(\wt\cdot\nabla_\Gamma v +(\Div_\Gamma\wt)v)u\,ds \\
 & = \frac12 \int_{\Gamma(t)} (\wt\cdot\nabla_\Gamma u v - \wt\cdot\nabla_\Gamma v u)\,ds
   -\frac12 \int_{\Gamma(t)} (\Div_\Gamma\wt) u v\,ds
\end{aligned}
\end{equation}
for $u,v \in \mathcal{V}(t)$. Note that we exploited $\bn=\nabla \phi/|\nabla \phi|$ on $\Gamma(t)$ and so  $\bn\cdot\nabla u=\bn\cdot \nabla v = 0$ here. We will use the characterization \eqref{e:intpartc} in our analysis and also  to define the finite element method.

\paragraph{Unique solvability}
To guarantee unique solvability in every time step, we ask for coercivity of the left-hand side bilinear form in \eqref{e:conttraceFEM1} with respect to $\Vert \cdot \Vert_{H^1(\G{n}{})}$. Testing  \eqref{e:conttraceFEM1} with $v=u^n$  and exploiting \eqref{e:intpartc} clarifies that a sufficient condition for coercivity is
\begin{equation} \label{e:xi}
  \Delta t \leq (2\xi)^{-1} \text{ with } \xi := \norm{ \Div_\Gamma (\bw -\frac12 \wt) }_{\infty}.
\end{equation}
 Using the notation $\kappa(t)=\Div_\Gamma\bn_\Gamma$ for the  mean curvature, we have
 $ \Div_\Gamma \bw = \Div_\Gamma \wt + \kappa(t) \wn$ and can also express condition \eqref{e:xi} with
 $$
 \xi =  \norm{ \frac12 \Div_\Gamma\wt + \kappa(t) \wn }_{\infty}.
 $$

\subsection{Stability of the semi-discrete method} \label{s:stab:semi-disc}
We now show a numerical stability bound for $u^n$. The goal of this paper is the study of a fully discrete method, but the treatment of the semi-discrete problem \eqref{e:ImEuler} gives some insight and serves for the purpose of better exposition.
From now on we \emph{assume that $\phi$ is the signed distance function} for $\Gamma(t)$ in $\O(\Gamma(t))$ for $t\in[0,T]$. Although
 this assumption is not essential, it simplifies our further (still rather technical) analysis. We assume $\Gamma(t)$ and its evolution are smooth such that $\phi\in C^2(\O(\Gs))$

Denote by $\bp(\bx,t)\,:\,\O(\Gamma(t))\to \Gamma(t)$ the closest point projection on $\Gamma(t)$. Then using that $\phi$ is the signed distance function  the second  equation in \eqref{e:ImEuler} can be written as $u^n(\bx)=u^n(\bp^n(\bx))$ in $\O(\G{n}{})$, $\bp^n(\bx)=\bp(\bx,t_n)$; and the passive advection of  $\Gamma$ by the velocity field yields
\begin{equation}\label{d_evol}
  \frac{\partial\phi}{\partial t}= - \wn\circ\bp\quad\text{in}~\O(\Gs);
\end{equation}
see \cite[Sect.~2.1]{DziukElliot2013a}.
We need the following result.
\smallskip

\begin{lemma}\label{l_est1} For $v\in L^2(\G{n-1}{})$ the following estimate holds:
\begin{equation}\label{est1}
  \|v\circ\bp^{n-1}\|_{\G{n}{}}^2 \le (1+c_{\ref{l_est1}}\Delta t)  \|v\|_{\G{n-1}{}}^2 \quad
\text{with} \quad  c_{\ref{l_est1}} = c (\norm{\wn}_{\infty,I_n} + \Delta t \norm{\nabla_\Gamma \wn}_{\infty,I_n})
\end{equation}
and a constant $c$ independent of $\Delta t$, $n$, $v$.
\end{lemma}
\begin{proof}
  For $\by\in \G{n-1}{}$ denote by $\kappa_i(\by)$, $i=1,2$, two principle curvatures, and let 
\begin{equation}\label{eq:kappa}
\kappa_i(\bx)=\kappa_i(\bp^{n-1}(\bx))\left[1+\phi^{n-1}(\bx)\kappa_i(\bp^{n-1}(\bx))\right]^{-1}\quad\bx\in\O(\G{n-1}{}).
\end{equation}
The surface measures on $\G{n-1}{}$ and $\G{n}{}$ satisfy, see, e.g., \cite[Proposition 2.1]{Demlow06},
\begin{equation}\label{aux2a}
\begin{split}
  \mu^n(\bx)d\bs^n(\bx)&= d\bs^{n-1}(\bp^{n-1}(\bx)),\quad \bx\in\G{n}{}, \quad \text{ with }\\
  \mu^n(\bx)&=
  (1-\phi^{n-1}(\bx)\kappa_1(\bx))
  (1-\phi^{n-1}(\bx)\kappa_2(\bx))
  \nabla\phi^n(\bx)^T\nabla\phi^{n-1}(\bx).
\end{split}
\end{equation}
We want to bound $| \mu^n(\bx) - 1|$.
Using $\phi^n(\bx)=0$ for $\bx\in\G{n}{}$ and \eqref{d_evol}, we get
\begin{equation}\label{aux2}
|\phi^{n-1}(\bx)|= |\phi^{n-1}(\bx)-\phi^n(\bx)|\le \norm{\wn}_{\infty,I_n}\Delta t.
\end{equation}
Using the smoothness of $\phi$, $|\nabla\phi|=1$ in $\O(\Gs)$ and \eqref{d_evol}, we also get with $\wne := \wn \circ \bp$
\begin{align}\label{aux3a}
  |1-\nabla\phi^n(\bx)^T\nabla\phi^{n-1}(\bx)|
  = \frac12 |\nabla\phi^n(\bx)-\nabla\phi^{n-1}(\bx)|^2
  \le \frac12 \sup_{t\in I_n}\norm{\nabla \wne}_{\infty,\O(\Gamma(t))}^2|\Delta t|^2.
\end{align}
We further note the identity, see, e.g., \cite[(2.2.16)]{Demlow06},
\[
\nabla \wne(\bx)=\left(\bI-\phi^{n-1}(\bx)\nabla^2\phi^{n-1}(\bx)\right)\nabla_\Gamma \wn(\bp^{n-1}(\bx)).
\]
From this and \eqref{aux3a} we conclude
\begin{equation}\label{aux3}
|1-\nabla\phi^n(\bx)^T\nabla\phi^{n-1}(\bx)|\le  \frac12 c \ \norm{ \nabla_\Gamma \wn }_{\infty,I_n}^2|\Delta t|^2,
\end{equation}
with a constant $c$ that depends only on the curvatures of $\Gamma$.
Now \eqref{aux2a},  \eqref{aux2} and \eqref{aux3} imply
\begin{equation*}\label{mu_n}
|1-\mu^n(\bx)|\le c_{\ref{l_est1}} \Delta t, \quad\text{for}~\bx\in\G{n}{},
\end{equation*}
 and so \eqref{est1} holds.
\end{proof}
\medskip
In the next lemma we show \emph{numerical stability} of the semi-discrete scheme.
\begin{lemma}\label{lem:conttraceFEM1}
  For $\{u^k\}_{k=1,\dots,N}$ the solution of \eqref{e:conttraceFEM1} with initial data $u^0 \in L^2(\G{0}{})$ there holds
\begin{equation}\label{stab_semi}
\|u^k\|_{\G{k}{}}^2+2\Delta t\nu\sum_{n=1}^{k}\|\nabla_\Gamma u^n\|_{\G{n}{}}^2 \le \exp(c_{\ref{lem:conttraceFEM1}} t_k)\|u^0\|_{\G{0}{}}^2,\quad \text{for}~k=0,\dots,N,
\end{equation}
for a constant $c_{\ref{lem:conttraceFEM1}}$ that is independent of $\Delta t$ and k.
\end{lemma}
\begin{proof}
We test \eqref{e:conttraceFEM1} with $2 u^n$ and apply \eqref{e:intpartc} to get
\[
\|u^n\|_{\G{n}{}}^2+\|u^n-u^{n-1}\|_{\G{n}{}}^2 + 2\Delta t\nu\|\nabla_\Gamma u^n\|_{\G{n}{}}^2  + 2\Delta t (\Div_\Gamma (\bw -\frac12 \wt)u^n,u^n)_{\G{n}{}}= \|u^{n-1}\|_{\G{n}{}}^2.
\]
Now we recall that the second equation in \eqref{e:ImEuler} implies  $u^{n-1}(\bx)=u^{n-1}(\bp^{n-1}(\bx))$ on $\G{n}{}$ and we use  \eqref{est1} for the right-hand side term; we also estimate the divergence term using the definition of $\xi$ in \eqref{e:xi}. This gives
\begin{equation}\label{aux4}
(1-2\xi\Delta t)\|u^n\|_{\G{n}{}}^2+2\Delta t\nu\|\nabla_\Gamma u^n\|_{\G{n}{}}^2 \le (1+c_{\ref{l_est1}} \Delta t)  \|u^{n-1}\|_{\G{n-1}{}}^2
\end{equation}
We sum up these inequalities for $n=1,\dots,k$, $k\le N$, and get
\[
\alpha \|u^k\|_{\G{k}{}}^2+2\Delta t\nu\sum_{n=1}^{k}\|\nabla_\Gamma u^n\|_{\G{n}{}}^2 \le \|u^0\|_{\G{0}{}}^2+(c_{\ref{l_est1}}+2\xi)\Delta t \sum_{n=0}^{k-1}  \|u^{n}\|_{\G{n}{}}^2,\quad \text{with}~\alpha=1-2\xi\Delta t>0.
\]
The quantities  $c$ and $\xi$ depend only on the PDE problem data such as $\bw$ and $\Gamma$, but not on numerical parameter $\Delta t$. In particular, one can always assume $\Delta t$ sufficiently small such that $\alpha>\frac12$.
Applying discrete Gronwall's inequality leads to the stability estimate \eqref{stab_semi}.
\end{proof}

\begin{remark}\rm \label{rem:expgrowth}
The stability estimate \eqref{stab_semi} admits exponential growth. This is rather natural, since the divergence term in \eqref{transport} is not sign definite and the concentration $u$ may grow exponentially if the (local) area of $\Gamma(t)$ shrinks when the surface evolves; see, e.g., analysis and a priori estimates in \cite{olshanskii2014eulerian}. The exponential growth does not happen if the divergence term is non-negative or if the tangential diffusion of $u$ is strong enough to suppress such growth; cf. Proposition~4.5 in  \cite{olshanskii2014eulerian}. Stability analysis may account for this phenomena by invoking conservation of total mass principle and the Friedrichs inequality,
\begin{equation*}\label{Fr}
\int_{\Gamma(t)} |\nabla_\Gamma u|^2 \, ds \geq c_F(t) \int_{\Gamma(t)} ( u- \frac{1}{|\Gamma(t)|} \bar u)^2 \, ds\quad\text{for all}~t\in[0,T],
\end{equation*}
with $c_F(t) >0$ and $\bar u(t):= \int_{\Gamma(t)} u(s,t)\, ds$. If no additional care is taken, the numerical method \eqref{e:ImEuler} conserves mass only approximately. One way to ensure total mass conservation for the numerical solution is to
introduce a Lagrange multiplier from $\mathbb{R}$ and to add the constraint $\bar u^n-\bar u^0=0$ to the system \eqref{e:ImEuler}.
The alternative is to augment  the left-hand side of \eqref{e:ImEuler} with the penalty term $\sigma(\bar u^n-\bar u^0)\bar v$, with an augmentation  parameter $\sigma\ge0$, as was done in \cite{olshanskii2014error} for the analysis of the space--time method.
The stabilizing term improves the mass conservation property and helps to make use of the Friedrichs inequality in the stability estimate. We skip the arguments here, which largely repeat the analysis above and the one in \cite{olshanskii2014error}.
These arguments bring one to the numerical stability estimate,
\begin{equation*}\label{stab_semi_no_exp}
\|u^k\|_{\G{k}{}}^2+\frac12\Delta t\nu\sum_{n=1}^{k}\|\nabla_\Gamma u^n\|_{\G{n}{}}^2 \le \|u^0\|_{\G{0}{}}^2+\frac12\Delta t\nu\|\nabla_\Gamma u^0\|_{\G{0}{}}^2+t_k\sigma|\bar u^0|^2,\quad \text{for}~k=0,\dots,N.
\end{equation*}
\end{remark}

Now we turn to the fully discrete case. Besides standard technical difficulties of passing from differential equations to algebraic and finite element functional spaces, we need to handle the situation, when the smooth surface $\G{n}{}$ is approximated by a set of piecewise smooth $\G{n}{h}$, $n=0,\dots,N$.

\section{Discretization in space and time} \label{s:FEM}
In order to reduce the repeated use of generic but unspecified constants,  further in the paper we write $x\lesssim y$ to state that the inequality  $x\le c y$ holds for quantities $x,y$ with a constant $c$, which is independent of the mesh parameters $h$, $\Delta t$, time instance $t_n$, and the position of $\Gamma$ over the background mesh. Similar we give sense to $x\gtrsim y$; and $x\simeq y$ will mean that both $x\lesssim y$ and $x\gtrsim y$ hold. However, we shall continue to monitor the explicit dependence of the estimate on the (norms of) normal surface velocity $\wn$.

\ifasslabel{
In the sequel, we will, at times, introduce conditions on discretization parameters. To distinguish these from other assumptions, e.g. assumptions on the PDE problem, we label them with an additional subscript; cf.
\eqref{eq:conddtkappa},
\eqref{cond1},
\eqref{cond4},
\eqref{cond2},
\eqref{cond3},
\eqref{cond5}, below.
}
\subsection{Fully discrete method} \label{s:fulldisc}
Assume a  family of consistent  subdivisions of $\Omega$ into shape regular tetrahedra. This constitutes our background  time-independent triangulations $\{\T_h\}_{h>0}$, with $\max\limits_{T\in\T_h}\mbox{diam}(T) \le h$.
$V_h$ denotes the bulk time-independent finite element space,
\begin{equation} \label{eq:Vh}
V_h:=\{v_h\in C(\Omega)\,:\, v_h|_S\in P_m(S), \forall S\in \mathcal{T}_h\},\quad m\ge1.
\end{equation}
Let  $\phi_h$ be a given continuous piecewise polynomial approximation (with respect to $\T_h$) of the level set function $\phi$ for all $t\in[0,T]$, which satisfies
\begin{equation}\label{phi_h}
\|\phi-\phi_h\|_{\infty,\Omega}+ h \|\nabla(\phi-\phi_h)\|_{\infty,\Omega}\lesssim \,h^{q+1},\quad \forall~t\in[0,T],
\end{equation}
with some $q\ge1$. For this estimate to hold, we assume that the level set function $\phi$ has the smoothness property $\phi \in C^{q+1}(\Omega)$.
Moreover, we assume that $\nabla \phi_h(\bx,t) \neq 0$ in a neighborhood around $\Gamma(t)$, $t\in[0,T]$ and that $\phi_h$ is sufficiently regular in time such that with $\phi_h^n(\bx) = \phi_h(\bx,t_n),~n=0,\dots,N$, there holds
\begin{subequations}\label{phi_hb}
\begin{align}\label{phi_hba}
\| \phi_h^{n-1}-\phi_h^n\|_{\infty,\Omega} & \lesssim\,\Delta t \|\wn\|_{\infty,I_n}, \\
\label{phi_hbb}
\|\nabla \phi_h^{n-1}-\nabla \phi_h^n\|_{\infty,\Omega} & \lesssim\,\Delta t\left( \|\wn\|_{\infty,I_n} +  \|\nabla \wn\|_{\infty,I_n}\right), \text{ for } n=1,\dots,N.
\end{align}
\end{subequations}
We define the discrete surfaces $\G{n}{h}$ approximating $\G{n}{}$ as the zero level of $\phi_h^n$,
\[
\G{n}{h}:=\{\bx\in\rr^3\,:\,\phi_h^n(\bx)=0\}.
\]
$\G{n}{h}$ is an approximation to $\G{n}{}$ with
\begin{equation}\label{eq:dist}
\operatorname{dist}(\G{n}{h},\G{n}{}) = \max_{x\in\G{n}{h}} |\phi^n(\bx)|
= \max_{x\in\G{n}{h}} |\phi^n(\bx) - \phi_h^n(\bx)| \leq \Vert \phi^n - \phi_h^n \Vert_{\infty,\Omega} \lesssim h^{q+1}.
\end{equation}
Furthermore, $\bn_h^n=\nabla\phi_h^n/|\nabla\phi_h^n|$ the normal vector to $\G{n}{h}$ and $ \bn^{n} = \nabla\phi^n$ the extended normal vector to $\G{n}{}$ satisfy for $\bx \in \G{n}{h}$
\begin{equation}\label{eq:normals}
  \vert \bn_h^n(\bx) - \bn^n(\bx) \vert \leq c | \nabla \phi_h^n(\bx) - \nabla \phi^n(\bx) | \lesssim h^q .
\end{equation}
In the following we assume that integrals on $\G{n}{h}$ can be computed accurately. In practice, this is only straightforward for piecewise linear $\phi_h^n$, i.e. $q=1$, while for higher order $\phi_h^n$ more care is needed, cf. Remark \ref{rem:integration} below.

The numerical method provides an extension of a finite element solution to a narrow band around   $\G{n}{h}$, which is defined as the union  of tetrahedra from
\begin{equation*}
{\mathcal{S}}(\G{n}{h}):=\{ S\in\mathcal{T}_h\,:| \phi_h^n(\bx) |\le\delta_n  \text{ for some } \bx \in S\},\quad{\O}(\G{n}{h})=
\text{int}\left({\bigcup}_{S\in{\mathcal{S}}(\G{n}{h})}\overline{S}\right),
\end{equation*}
where
\begin{equation} \label{e:delta}
  \delta_n := c_\delta \| \wn\|_{\infty,I_n}~\Delta t
\end{equation}
is the minimum thickness of the extension layer and $c_\delta\geq 1 $ is an $\O(1)$ mesh-independent constant.
Recall that $\phi_h$ is an approximate distance function, so that
$U_{\delta_n}(\G{n}{h}):=\{\bx\in\mathbb{R}^3\,:\, |\phi_h^n(\bx)|<\delta_n\} \subset \O(\G{n}{h})$ describes a discrete tubular neighborhood to $\G{n}{h}$. We refer to Figure \ref{fig:discrete domains} for a sketch.

Further, we require that $\delta_n \leq c$ for a constant $c$ that only depends on the temporal resolution of the surface dynamics and the roughness of the surface.
We assumed that the surface is smooth at all time so that $\norm{ \kappa }_{\infty,I_n} \lesssim 1$.
Hence, we formulate the following condition on the time step size:
\begin{equation} \label{eq:conddtkappa}\asslabel
  \Delta t \leq c_{\bl{}} (c_\delta \norm{\wn}_{\infty,I_n})^{-1},  ~ n = 1,\dots,N,
\end{equation}
with some $c_{\bl{}}$ sufficiently small, but independent of $h$, $\Delta t$ and $n$.

We also denote by
$\T_{\Gamma}^n$ the set of elements intersected by $\G{n}{h}$,
\begin{equation*}
\T_{\Gamma}^n:=\{ S\in\mathcal{T}_h\,:\, \mathcal{H}_2(S\cap\G{n}{h})>0\}\quad\text{and}~~\quad{\O}_\Gamma(\G{n}{h}):=
\text{int}\left({\bigcup}_{S\in\T_{\Gamma}^n}\overline{S}\right),
\end{equation*}
where $\mathcal{H}_2$ denotes the two-dimensional Hausdorff measure.
We assume that $\O(\G{n}{})$ is such that
\begin{equation}\label{neigh_cond}
  \O(\G{n}{h})\subset\O(\G{n}{})\quad\text{and}\quad \O_\Gamma(\G{n+1}{h})\subset\O(\G{n}{}).
\end{equation}
Note that \eqref{phi_h}, \eqref{phi_hba} and \eqref{e:delta} ensure that
\begin{equation}\label{cond1} \asslabel
  c_\delta \text{ sufficiently large implies }
\G{n}{h} \subset U_{\delta_{n-1}}(\G{n-1}{h}) \text{ and } 
{\O_\Gamma}(\G{n}{h})\subset{\O}(\G{n-1}{h}).
\end{equation}
This condition is the discrete analog of \eqref{ass1} and it is essential for the well-posedness of the method.
\begin{figure}
  \begin{center}
    \begin{tikzpicture}[scale=3.4]
      \input{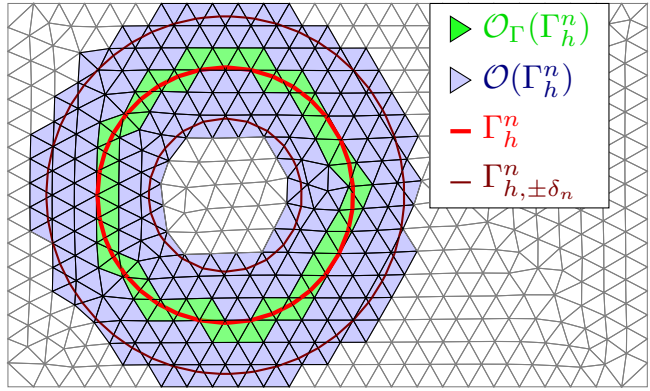}
      \draw[fill=white] (0.9,0.75) -- (1.6,0.75) -- (1.6,-0.05) -- (0.9,-0.05) -- cycle;
      \node[right,scale=1.25] at (1.06,0.65) {\color{green!85!black} $\O_{\Gamma}(\G{n}{h})$};
      \draw[fill=green!85] (1.06,0.65) -- (0.98,0.7) -- (0.98,0.6) -- cycle ;
      \node[right,scale=1.25] at (1.06,0.45) {\color{blue!50!black} $\O(\G{n}{h})$};
      \draw[fill=blue!20] (1.06,0.45) -- (0.98,0.5) -- (0.98,0.4) -- cycle ;
      \node[right,scale=1.25] at (1.06,0.25) {\color{red} $\G{n}{h}$};
      \draw[ultra thick,red] (1.06,0.25) -- (0.98,0.25);
      \node[right,scale=1.25] at (1.06,0.05) {\color{red!50!black} $\G{n}{h,\pm \delta_n}$};
      \draw[thick,red!50!black] (1.06,0.05) -- (0.98,0.05);
    \end{tikzpicture}
  \end{center}
  \vspace*{-0.2cm}
  \caption{Sketch of discrete domains and interfaces.}
  \vspace*{-0.2cm}
  \label{fig:discrete domains}
\end{figure}
We define finite element spaces
\begin{equation}
V_h^n=\{v \in C({\O}(\G{n}{h}))\,:\, v\in P_m(S), \forall S\in {\mathcal{S}}(\G{n}{h})\},\quad m\ge1.
\end{equation}
These spaces are the restrictions of the time-independent bulk space $V_h$ on all tetrahedra from ${\mathcal{S}}(\G{n}{h})$.

The numerical method is based on the semi-discrete formulation \eqref{e:conttraceFEM1} and identity \eqref{e:intpartc}. It reads:
For a given $u_h^0 \in V_h^0$ find  $u_h^n\in V_h^n$, $n=1,\dots,N$, satisfying
\begin{multline}\label{e:traceFEM1}
\int_{\G{n}{h}}\left\{\frac{u_h^n-u_h^{n-1}}{\Delta t} v_h + \frac12 (\wte\cdot\nabla_{\G{}{h}} u^n_h v_h - \wte\cdot\nabla_{\G{}{h}} v_h u_h^n) +\Div_{\G{}{h}}(\bw^e - \frac12 \wte) u_h^n v_h\, \right\}ds_h\\ +\nu\int_{\G{n}{h}}\nabla_{\G{}{h}}u^n_h\cdot\nabla_{\G{}{h}}
v_h\, ds_h+\stab \int_{\O(\G{n}{h})}(\bn_h^n\cdot\nabla u_h^n) (\bn_h^n\cdot\nabla v_h) d\bx =0,
\end{multline}
for all $v_h\in V_h^n$. Here $\bn_h=\nabla\phi_h^n/|\nabla\phi_h^n|$ in $\O(\G{n}{h})$, $\stab>0$ is a  parameter, $\bw^e(\bx)=\bw(\bp^n(\bx))$ is lifted data on $\G{n}{h}$ from $\G{n}{}$. The first term in \eqref{e:traceFEM1} is well-defined thanks to condition \eqref{cond1}.
As we discussed in the introduction,   the term $\stab\int_{\O(\G{n}{h})}(\bn_h\cdot\nabla u_h^n) (\bn_h\cdot\nabla v_h) d\bx$ plays several roles.
We shall see that for $\stab$ not too small, it ensures the form on the left hand side to be elliptic on $V_h^n$, rather than only on the space of traces. Therefore, on each time step we obtain a FE solution defined in  $\O(\G{n}{h})$ (this can be seen as an implicit extension procedure).
Furthermore, it stabilizes the problem algebraically, i.e. the resulting systems of algebraic equations are well-conditioned, see section~\ref{s:algebra}.

\begin{remark}[Numerical integration]\label{rem:integration}\rm
  The discrete surface $\G{n}{h}$ is described only implicitly via the zero-level of a discrete level set function. In general, it is a non-trivial task to obtain a parametrized representation of $\G{n}{h}$ which would allow for a straightforward application of numerical quadrature rules. On simplices and in the low order case where $\phi_h^n$ is a piecewise linear approximation of the level set function $\phi^n$ ($q=1$ in \eqref{phi_h}), an explicit reconstruction of $\G{n}{h}$ is easily available, cf. e.g. \cite{mayer2009interface}. On hyperrectangles, a low order case where the accuracy of the implicit representation is $q=1$ in \eqref{phi_h} can be dealt with a marching cube \cite{lorensen1987marching} approximation. However, the higher order case $q>1$ is more involved and requires special approaches for the construction of quadrature rules. We do not extend this discussion here but refer to the literature instead, cf. \cite{fries2015,lehrenfeld2015cmame,muller2013highly,olshanskii2016numerical,saye2015hoquad,sudhakar2013quadrature}.
\end{remark}

\section{Analysis of the fully discrete method} \label{s:Analysis}
In this section we carry out the numerical analysis of the fully discrete method.
Before we can perform the stability and consistency analysis (subsections
\ref{s_stab} and \ref{s:consistency}) to derive a priori error bounds in subsection \ref{sec:aprioriest}, we require two results that are technically more involved.
The first one gives control in the $L^2$ norm in a narrow band volume based on
a combination of the $L^2$ norm on the surface and the normal gradient in the volume that is provided by the stabilization. The result is treated in subsection \ref{sec:volumecontrol} and is a generalization of a result from \cite{burman2016cut} which is also found in \cite[Lemma 7.6]{grande2016analysis}.
The second result provides bounds for the evaluation of a lifting of a function that is naturally defined on $\G{n-1}{h}$ to $\G{n}{h}$. The result is the counterpart to Lemma~\ref{l_est1} on the discrete level and is treated in subsection \ref{sec:stabshift}.

\subsection{Volume control by the normal diffusion stabilization} \label{sec:volumecontrol}
Before we can state and prove the lemma on the normal diffusion stabilization we need some preparation.

We denote the limiting level sets of $\phi_h^n$ with $|\phi_h^n|=\delta_n$ as $\G{n}{h,\pm\delta_n} := \{ \phi_h^n(\bx) = \pm \delta_n\}$. The corresponding set of elements cut by $\G{n}{h,\pm\delta_n}$ is denoted by $\O_\Gamma(\G{n}{h,\pm\delta_n})$, cf. Figure \ref{fig:discrete domains}.
Corresponding to $\G{n}{h,-\delta_n}$ and $\G{n}{h,\delta_n}$ we recall the neighborhood $U_{\delta_n}(\G{n}{h})=\{\bx\in\mathbb{R}^3\,:\, |\phi_h^n(\bx)|<\delta_n\} \subset \O(\G{n}{h})$.
Now, we introduce a mapping $\Phi: \O(\G{n}{h}) \rightarrow \O(\G{n}{})$ that allows to map from approximated level sets to exact level sets. For $\bx \in \O(\G{n}{h})$ we define $\Phi(\bx):=\bx+(\phi_h^n(\bx)-\phi^n(\bx))\bn^n(\bx)$ with $\bn^n(\bx) = \nabla \phi^n (\bx) = \bn^n(\bp^n(\bx))$ which has
 $\phi^n \circ \Phi = \phi_h^n$ in $\O(\G{n}{h})$, i.e.
\begin{equation} \label{eq:map}
  \phi^n(\bx+(\phi_h^n(\bx)-\phi^n(\bx))\bn^n(\bx))=\phi^n(\bx) + \phi_h^n(\bx)-\phi^n(\bx) = \phi_h^n(\bx) \quad \forall~\bx \in \O(\G{n}{h}).
\end{equation}
\vspace*{-0.2cm}
\begin{lemma}
The mapping $\Phi$ is well-defined, continuous and $\Phi|_S \in C^{q+1}(S)$ for any $S\in\mathcal{S}(\G{n}{h})$. There hold
  $\Phi(\G{n}{h}) = \G{n}{}$ and
    \begin{align} \label{eq:estPhi}
    \Vert \Phi - \operatorname{id} \Vert_{\infty,\O(\G{n}{h})}
    \lesssim h^{q+1}, \quad
    \Vert D \Phi - I \Vert_{\infty,\O(\G{n}{h})}
    \lesssim h^{q}.
    \end{align}
    Further, for $h$ sufficiently small $\Phi$ is invertible.
\end{lemma}
\begin{proof}
  The smoothness is obtained by construction. To see $\Phi(\G{n}{h}) = \G{n}{}$ we recall that $\phi^n \circ \Phi = \phi_h^n$ holds also for $\bx \in \G{n}{h} = \{ \phi_h^n= 0\}$ which implies that $\Phi(\bx) \in \G{n}{} =  \{ \phi^n = 0\}$.
Finally,  \eqref{eq:estPhi} follows from \eqref{phi_h}.
\end{proof}

We use this mapping to map from the discrete surface to the exact one. We introduce the following notation. For $u \in V_h^n$ we define $\tilde{u} := u \circ \Phi^{-1}$, $\tilde{\O}(\G{n}{h,\ast}) := \Phi(\O(\G{n}{h,\ast}))$, $\tilde{\O}_\Gamma(\G{n}{h,\ast}) := \Phi(\O_\Gamma(\G{n}{h,\ast}))$
for $\G{n}{h,\ast} \in \{ \G{n}{h,-\delta_n},\G{n}{h},\G{n}{h,\delta_n}\}$,
$\G{}{\pm \delta_n} := \Phi(\G{n}{h,\pm \delta_n}) = \{ \phi(\bx) = \pm \delta_n
\}$ and
$U_{\delta_n}(\G{n}{}):=\{\bx\in\mathbb{R}^3\,:\, |\phi^n(\bx)|<\delta_n\} = \Phi(U_{\delta_n}(\G{n}{h}))$.
Due to \eqref{eq:estPhi} we have that
\begin{equation}\label{eq:equivvol}
  \Vert \tilde{u} \Vert_{\tilde{\O}(\G{n}{h})}^2 = \int_{\tilde{\O}(\G{n}{h})} \tilde{u}^2~d \bx = \int_{\O(\G{n}{h})} \underbrace{\operatorname{det}(D\Phi)}_{\simeq 1} u^2~d \bx \simeq \Vert u \Vert_{\O(\G{n}{h})}^2
\end{equation}
and similarly one easily shows (see, e.g., \cite[Lemma 3.7]{LR_ARXIV_2016a})
\begin{equation}\label{eq:equivsurf}
  \Vert u \Vert_{\G{n}{h}}^2 \simeq \Vert \tilde{u} \Vert_{\G{n}{}}^2.
\end{equation}

\begin{lemma}\label{lem:onelayer}
On a quasi-uniform family of triangulations, for sufficiently small $h$, for $\tilde{u} \in V_h^n \circ \Phi^{-1}$ there holds for $\G{n}{h,\ast} \in \{ \G{n}{h,-\delta_n},\G{n}{h},\G{n}{h,\delta_n} \}$ with $\G{n}{\ast} = \Phi(\G{n}{h,\ast})$
\begin{equation}\label{eq:estonelayer}
  \Vert \tilde{u} \Vert_{\tilde{\O}_\Gamma(\G{n}{h,\ast})}^2
  \lesssim h \Vert \tilde{u} \Vert_{\G{n}{\ast}}^2
  + h^2 \Vert \bn^n \cdot \nabla \tilde{u} \Vert_{\tilde{\O}_\Gamma(\G{n}{h,\ast})}^2.
\end{equation}
\end{lemma}
\begin{proof}
  The technical proof is given in \cite[section 7.2]{grande2016analysis}. The main idea is the application of the co-area formula combined with estimates along paths which are normal to the interfaces $\G{n}{\ast}$ and cross the interfaces $\G{n}{\ast}$. Below in Theorem \ref{lemcrucial} we apply similar techniques.
\end{proof}

\begin{theorem} \label{lemcrucial}
  For $h$ sufficiently small and $\Delta t$ so that \eqref{eq:conddtkappa} is fulfilled and
$\delta_n$ as in \eqref{e:delta}, the following uniform with respect to $\delta_n$, $h$ and $n$ estimates holds for any $u \in V_h^n$
\begin{subequations}
  \begin{align}\label{fund1a}
\|u\|_{U_{\delta_n}(\G{n}{h})}^2 & \lesssim & \hspace*{-1.5cm}\delta_n \|u\|_{\G{n}{h}}^2 &+& \hspace*{-1.5cm}\delta_n^2 \|\bn_h^n \cdot \nabla u\|_{\O(\G{n}{h})}^2,
 \\
    \label{fund1}
\|u\|_{\O(\G{n}{h})}^2 & \lesssim & \hspace*{-1.5cm}(\delta_n + h) \|u\|_{\G{n}{h}}^2 &+& \hspace*{-1.5cm}(\delta_n + h)^2 \|\bn_h^n \cdot \nabla u\|_{\O(\G{n}{h})}^2.
\end{align}
\end{subequations}
\end{theorem}
\begin{proof}[Sketch of the proof]
  We only sketch the proof here. A complete proof is given in
  \ifarxiv
  the appendix.
  \else
  \cite[Appendix]{arXiv}.
  \fi
Based on the co-area formula on the smooth mapped domains, for $\tilde{u} = u \circ \Phi^{-1}$ there holds
\begin{equation*}
    \Vert \tilde{u} \Vert_{U_{\delta_n}(\Gamma^n)}^2
 \lesssim \delta_n \Vert \tilde{u} \Vert_{\Gamma^n}^2 + \delta_n^2   \Vert \bn^n \cdot \nabla \tilde{u} \Vert_{U_{\delta_n}(\Gamma^n)}^2.
\end{equation*}
Combining this estimate with the result of Lemma \ref{lem:onelayer} for $\G{n}{\pm \delta_n}$ and the overlapping decomposition $\tilde{\O}_\Gamma(\G{n}{h,\pm \delta_n}) \cup U_{\delta_n}(\Gamma^n) = \tilde{\O}(\G{n}{h})$ we arrive at
  \begin{equation*}
    \Vert \tilde{u} \Vert_{\tilde{\O}(\G{n}{h})}^2
 \lesssim  (h+\delta_n)  \Vert \tilde{u} \Vert_{\Gamma^n}^2 +  (h + \delta_n)^2   \Vert \bn^n \cdot \nabla \tilde{u} \Vert_{\tilde{\O}(\G{n}{h})}^2.
\end{equation*}
Finally, incorporating geometrical errors for the normal, $\bn^n \neq \bn_h^n$, and applying the equivalence of norms on the mapped domains yields the result.
\end{proof}
\subsection{Stability of shift operations and the normal diffusion stabilization} \label{sec:stabshift}
To control the effect of the geometric error, we require the following mild restriction in our analysis,
\begin{equation}\label{cond4}\asslabel
  h^{2q} \leq c_{h} \Delta t,
\end{equation}
for some $c_{h}$ independent of $\Delta t$ and $h$.
We recall that $q\ge 1$ is defined in \eqref{phi_h}.

Now, we turn our attention to a discrete analogue of Lemma~\ref{l_est1} which we prove in Lemma \ref{lem2a}.
\begin{lemma} \label{lem2a} For $v \in L^2(\G{n-1}{})$, $\Delta t$ so that \eqref{eq:conddtkappa} is fulfilled 
with sufficiently small $c_{\bl{}}$
and $h$ such that \eqref{cond4} is fulfilled, it holds
 \begin{equation} \label{fund2a}
   \|v\circ\bp^{n-1}\|_{\G{n}{h}}^2
   \le
   (1+c_{\ref{lem2a}} \Delta t) \ \|v\circ\bp^{n-1}\|_{\G{n-1}{h}}^2
\end{equation}
with some $c_{\ref{lem2a}}$ independent of $h$, $\Delta t$, and $n$.
\end{lemma}
\begin{proof}
For $n = 1,\ldots,N$ and $k=n-1,n$ we define the lift operator from $\G{n}{}$ to $\G{k}{h}$,
\begin{equation} \label{eq:lift} \asslabel
  \blk{n}{k}: \G{n}{} \to \G{k}{h}, \quad \blk{n}{k}(\bx) = \bx + d^n(\bx) \bn^n(\bx),
\end{equation}
where $d^n(\bx)\in \rr$ is the smallest (in absolute value) value so that $\bx + d^n(\bx) \bn^n(\bx) \in \G{k}{h}$. For $k=n$ we also write $\bl{n} = \blk{n}{n}$.
These liftings are well-defined bijection mappings if $h$  and $c_{\bl{}}$ in \eqref{eq:conddtkappa} are sufficiently small.

  For $\bx\in\G{n-1}{}$, we make use of the lift operators $\blk{n-1}{k}\,:\,\G{n-1}{}\to \G{k}{h}$ such that $\blk{n-1}{k}(\bp^{n-1}(\bx))=\bx$ on $\G{k}{h}$, $k=n-1,n$. For $\bx\in\G{n-1}{}$, denote $\bx^{k}=\blk{n-1}{k}(\bx)\in\G{k}{h}$, $k=n-1,n$, cf. Figure \ref{fig:lift}.
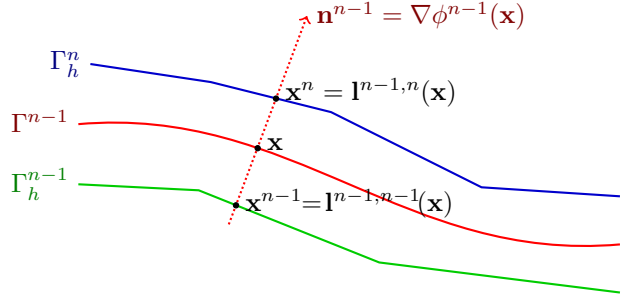
\begin{figure}
  \begin{center}
    \begin{tikzpicture}[scale=0.8]
      \draw[red,thick] (1,3) to[out=5,in=185] (10,1);
      \draw[green!80!black,thick] (1,2) -- (3,1.9) -- (6,0.7) -- (10,0.2);
      \draw[blue!80!black,thick] (1.2,4) -- (3.2,3.7) -- (5.2,3.2) -- (7.7,1.95) -- (10,1.8);
      \draw[->,red,thick, densely dotted] (3.5,1.333) -- (4.8,4.8);
      \node[right] at (4.8,4.8) {\color{red!50!black}$\bn^{n-1} = \nabla \phi^{n-1}(\mathbf{x})$};
      \filldraw (3.98,2.6) circle (1.25pt);
      \filldraw (3.62,1.65) circle (1.25pt);
      \filldraw (4.285,3.425) circle (1.25pt);
      \node[right] at (3.98,2.67) {$\bx$};
      \node[right] at (4.285,3.525) {$\bx^n = \blk{n-1}{n}(\bx)$};
      \node[right] at (3.62,1.68) {$\bx^{n-1}\!\!=\!\blk{n-1}{n-1}\!(\bx)$};
      \node[left] at (1,2) {\color{green!50!black} $\G{n-1}{h}$};
      \node[left] at (1.2,4) {\color{blue!50!black} $\G{n}{h}$};
      \node[left] at (1,3) {\color{red!50!black} $\G{n-1}{}$};
    \end{tikzpicture}
  \vspace*{-0.2cm}
  \end{center}
  \caption{Sketch of the geometries in the proof of Lemma \ref{lem2a}.}
  \vspace*{-0.2cm}
  \label{fig:lift}
\end{figure}
For the ratio of surface measures on $\G{k}{h}$, $k=n-1,n$, and $\G{n-1}{}$, so that
\begin{equation}\label{aux2b}
  \mu^k_h(\bx^k)d\bs^k_h(\bx^k) = d\bs^{n-1}(\bx),\quad \bx\in\G{n-1}{},
\end{equation}
there holds, cf.
\ifarxiv
Lemma \ref{lem:surfmeasratio} in the appendix,
\else
\cite[Appendix]{arXiv},
\fi
\begin{equation*}
  |1-\mu^n_h(\bx^n)/\mu^{n-1}_h(\bx^{n-1})| \lesssim
  \ifarxiv
  c_{\ref{lem:surfmeasratio}}
  \else
  c_{\mu}
  \fi
  \Delta t.
\end{equation*}
Transformation of the integrals on $\G{n}{h}$ and $\G{n-1}{h}$ to $\G{n-1}{}$ concludes the proof.
\end{proof}
We require an analogue to Lemma \ref{lem2a}, where a discrete normal gradient in the volume is used to replace the closest point projection.
This lemma needs some preparatory results, the following Lemmas \ref{SQ} and \ref{Ud}.
\begin{lemma}\label{SQ}  For $S\in\T_h$,   let  $Q\subset S$ be a subdomain of Lebesgue measure $|Q|$. Then it holds,
\begin{equation}\label{SQ_ineq}
\|f\|_{L^2(Q)}\le c\,(|Q|/|S|)^{\frac{1}{2}} \|f\|_{L^2(S)}, \quad \forall~f\in \mathcal{P}_l(S),~l\ge0,
\end{equation}
with a constant $c$, which is independent of $f$, $S$ and $Q$, but may depend on $l$ and the minimal angle condition in $\T_h$.
\end{lemma}
\begin{proof}
  Let $\psi\,:\,S\to\widehat{S}$ be the affine mapping to the reference simplex in $\mathbb{R}^d$; $\widehat{Q}=\psi(Q)$.
For $\hat{f}=f\circ\psi^{-1}\in\mathcal{P}_l(\widehat{S})$, we have due to norm equivalence in finite dimensional spaces
\begin{equation*}
\begin{aligned}
\|\hat f\|_{L^2(\widehat{Q})}&\le |\widehat{Q}|^{\frac12}\|\hat f\|_{L^\infty(\widehat{Q})} \le |\widehat{Q}|^{\frac12}\|\hat f\|_{L^\infty(\widehat{S})}
\le c|\widehat{Q}|^{\frac12}\|\hat{f}\|_{L^2(\widehat{S})}
\end{aligned}
\end{equation*}
where $c$ is independent of $f$, $S$ and $Q$.
Standard arguments, i.e. changing the domain of integration, using  the maximum angle condition and noting
$|\widehat{Q}|\simeq |\widehat{Q}|/|\widehat{S}|=|{Q}|/|{S}|$ complete the proof.
\end{proof}

We note in passing that the proof of the lemma obviously holds for any finite dimensional space $V$ on $S$ (instead of $\mathcal{P}_l(S)$) such that $V=\psi^{-1}(\widehat{V})$, where $\widehat{V}$ is  a fixed ($S$-independent) finite dimensional subspace  of $L^\infty(\widehat{S})$.
  We apply the result of the above to arrive at the following lemma.
  \begin{lemma}\label{Ud} For all $u \in V_h^{n}$, $n=1,..,N$, there holds
\begin{subequations}\label{SQ_esta}
\begin{align}
\|\nabla u\|_{U_{\delta_n}(\Gamma_h^n)}^2&\lesssim \delta_n(\delta_n+h)^{-1} \|\nabla u\|_{\O(\G{n}{h})}^2,\\
\|\bn_h^n \cdot \nabla u\|_{U_{\delta_n}(\Gamma_h^n)}^2&\lesssim \delta_n(\delta_n+h)^{-1} \|\bn_h^n \cdot \nabla u\|_{\O(\G{n}{h})}^2.
\end{align}
\end{subequations}
\end{lemma}
\begin{proof}
  First, we note that $\delta_n (\delta_n + h)^{-1} \simeq \min\{ \delta_n/h,1\}$. Correspondingly, we distinguish the cases $\delta_n/h \leq c $ and $\delta_n/h > c$ for a fixed small constant $c$. If $\delta_n/h > c $ we have $\delta_n (\delta_n + h)^{-1} \simeq 1$ so that the result is obvious. Hence, we consider $\delta_n/h \leq c$. Due to shape regularity and resolution of the surface $\Gamma^n$ by the mesh, we have that for every $S \in {\mathcal{S}}(\G{n}{h})$ with $Q_S = U_{\delta_n}(\G{n}{h}) \cap S$ there holds $|Q_S| \lesssim 2 h^2 \delta_n$ and $|S| \gtrsim h^3$. Hence, for every polynomial $p \in \mathcal{P}_l(S)$ there holds with Lemma \ref{SQ} $\norm{p}_{Q_S}^2 \lesssim \delta_n/h \norm{p}_{S}^2$ with a constant that is independent of $S \in {\mathcal{S}}(\G{n}{h})$.
As $\nabla u|_S$ and $ (\bn_h \cdot \nabla u)|_S$ are polynomials of fixed degrees we can apply this result element by element (with a uniform constant) which concludes the proof.
\end{proof}

\begin{lemma} \label{lem2} Under the conditions of Lemma \ref{lem2a} the following estimate holds for all $v_h \in V_h^{n-1}$,
 \begin{equation} \label{fund2}
   \|v_h\|_{\G{n}{h}}^2
   \le
    (1+c_{\ref{lem2}a}\Delta t)
    \|v_h\|_{\G{n-1}{h}}^2 + c_{\ref{lem2}b} \delta_{n-1} (\delta_{n-1}+h)^{-1} \|\bn_h^{n-1} \cdot \nabla v_h\|_{\O(\G{n-1}{h})}^2,
\end{equation}
for some $c_{\ref{lem2}a}$ and $c_{\ref{lem2}b}$ independent of $h$, $\Delta t$ and $n$.
\end{lemma}
\begin{proof} We first note that $\G{n}{h}\subset U_{\delta_{n-1}}(\G{n-1}{h})\subset\O(\G{n-1}{h})$, cf. condition \eqref{cond1}.
From conditions  \eqref{cond1} and \eqref{neigh_cond} we know that both $\G{n}{h}$ and $\G{n-1}{h}$ are in $\O(\G{n-1}{})$. Hence, we can define
a lift $v^\ell\in L^2(\G{n}{h})$ for  $v\in L^2(\G{n-1}{h})$ along normal directions to $\G{n-1}{}$, i.e.,
$v^\ell(\bx)=v(\blk{n-1}{n} \bp^{n-1}(\bx))$, $\bx\in\G{n}{h}$.
We start with the splitting
$$
\|v_h\|_{\G{n}{h}}^2 = \int_{\G{n}{h}}(|v_h|^2-|v_h^\ell|^2)\,d\bx+\|v_h^\ell\|_{\G{n}{h}}^2
$$
and bound the first term on the right-hand side (we abbreviate $U_{\delta_{n-1}} = U_{\delta_{n-1}}(\G{n-1}{h})$ here):
  \begin{align*}
    &
      \int_{\G{n}{h}}( |v_h|^2-|v_h^\ell|^2)\,d\bs
      \lesssim  \int_{U_{\delta_{n-1}}} \left| \bn^{n-1}\cdot\nabla(|v_h|^2-|v_h^\ell|^2) \right| \,d\bx
    &
      ({\footnotesize |v_h|^2=|v_h^\ell|^2~\text{on}~\G{n-1}{h}} )
    \\
    \leq
    &
      \, \int_{U_{\delta_{n-1}}} \left| \bn_h^{n-1}\cdot\nabla|v_h|^2 \right| \,d\bx
      +\, \int_{U_{\delta_{n-1}}} \left| (\bn^{n-1}-\bn_h^{n-1})\cdot\nabla|v_h|^2 \right| \,d\bx
    &
      ({\footnotesize \text{as } \bn^{n-1}\cdot\nabla |v_h^\ell|^2\!\! = 0} )
    \\
    \lesssim
    &
      \|\bn_h^{n-1}\cdot\nabla v_h\|_{{U_{\delta_{n-1}}}}\|v_h\|_{U_{\delta_{n-1}}}
      +\|\bn^{n-1}-\bn_h^{n-1}\|_{\infty,\O(\G{n-1}{h})} \|\nabla v_h\|_{U_{\delta_{n-1}}}\|v_h\|_{U_{\delta_{n-1}}}
    \\
    \lesssim
    &
      \|\bn_h^{n-1}\cdot\nabla v_h\|_{U_{\delta_{n-1}}}^2 +   \|v_h\|_{U_{\delta_{n-1}}}^2 +\frac{h^{q}\delta_{n-1}^{\frac12}}{(\delta_{n-1}+h)^{\frac12}} \|\nabla v_h\|_{\O(\G{n-1}{h})}\|v_h\|_{U_{\delta_{n-1}}}
    &
      ({\footnotesize \text{\eqref{e:delta} and Lem.\ref{Ud}}})
    \\
    \lesssim
    &
      \|\bn_h^{n-1}\cdot\nabla v_h\|_{U_{\delta_{n-1}}}^2 +   \|v_h\|_{U_{\delta_{n-1}}}^2 +\frac{h^{q-1}\delta_{n-1}^{\frac12}}{(\delta_{n-1}+h)^{\frac12}} \| v_h\|_{\O(\G{n-1}{h})}\|v_h\|_{U_{\delta_{n-1}}}
    &
      ({\footnotesize \text{FE inv. ineq.}})
    \\
    \lesssim
    &
      \frac{\delta_{n-1}}{(\delta_{n-1}+h)}  \|\bn_h^{n-1}\cdot\nabla v_h\|_{\O(\G{n-1}{h})}^2
      + \frac{\delta_{n-1}}{(\delta_{n-1}+h)} \| v_h\|_{\O(\G{n-1}{h})}^2+ \|v_h\|_{U_{\delta_{n-1}}}^2
    &
      ({\footnotesize\text{Lem.\ref{Ud}, } h^{q-1}\lesssim 1})
    \\
    \lesssim
    &
      \frac{\delta_{n-1}}{(\delta_{n-1}+h)}  \|\bn_h^{n-1}\cdot\nabla v_h\|_{\O(\G{n-1}{h})}^2
      + \delta_{n-1}  \| v_h\|_{\G{n-1}{h}}^2
    &
      ({\footnotesize\text{Thm.~\ref{lemcrucial}}}).
\end{align*}
Next, we apply Lemma \ref{lem2a},
$\|v_h^\ell\|_{\G{n}{h}}^2 \leq  (1+ c_{\ref{lem2a}} \Delta t) \  \|v_h\|_{\G{n-1}{h}}^2$
so that we obtain
\begin{align*}
  \|v_h\|_{\G{n}{h}}^2
  \leq
 (1+\underbrace{c_{\ref{lem2a}} \Delta t + c \delta_{n-1}}_{c_{\ref{lem2}a} \Delta t}) \| v_h\|_{\G{n-1}{h}}^2
 + c_{\ref{lem2}b} \delta_{n-1}(\delta_{n-1}+h)^{-1} \|\bn_h^{n-1}\cdot\nabla v_h\|_{\O(\G{n-1}{h})}^2.
\end{align*}
\end{proof}

\subsection{Stability analysis} \label{s_stab}
For the well-posedness and numerical stability we need some additional conditions on the discretization parameters.
First, we formulate a condition on the time step size analogously to \eqref{e:xi}:
\begin{equation}\label{cond2}\asslabel
  \Delta t \le (4\xi_h)^{-1} \text{ with } \xi_h := \max_{n=0,..,N} \Vert \Div_{\G{}{h}} (\bw^e -\frac12 \wte) \Vert_{{\infty},\G{}{h}(t_n)}.
\end{equation}
From the definition of $\xi_h$ and geometrical approximation condition \eqref{phi_h}, it follows that
\begin{equation}\label{ass2}
\xi_h\le C_0,
\end{equation}
with some $C_0$ independent of $\Delta t$ and $h$.

For the stability analysis we formulate the \emph{lower} bound condition on  $\stab$:
\begin{equation}\label{cond3}\asslabel
  \stab \ge c_{\delta} c_{\ref{lem2}b} \Vert \wn \Vert_{\infty,I_n} (\delta_n+h)^{-1},
\end{equation}
with $c_{\ref{lem2}b}$ as in Lemma \ref{lem2} (which is a constant independent of $\Delta t$ and $h$). This condition and \eqref{e:delta} imply $ \stab \Delta t \geq c_{\ref{lem2}b} \delta_n (\delta_n + h)^{-1}$.

With conditions \eqref{cond4} and \eqref{cond3} fulfilled,
estimate \eqref{fund2} simplifies to
 \begin{equation} \label{fund2simple}
   \|v_h\|_{\G{n}{h}}^2
   \le
    (1+c_{\ref{lem2}a} \Delta t)
    \|v_h\|_{\G{n-1}{h}}^2 + \stabold \Delta t \|\bn_h^{n-1} \cdot \nabla v_h\|_{\O(\G{n-1}{h})}^2 \quad \forall\,v_h \in V_h^{n-1}.
\end{equation}

For the notation convenience,  we introduce the bilinear form,
\begin{equation} \label{e:def:an}
  \begin{split}
    a_n(u,v) := &
    \int_{\G{n}{h}} \left( \frac12 (\wte\cdot\nabla_{\G{}{h}} u) v - \frac12 (\wte\cdot\nabla_{\G{}{h}} v) u  +(\Div_{\G{}{h}} (\bw^e - \frac12 \wte)) u v\right) \, ds \\
    & +\nu\int_{\G{n}{h}}(\nabla_{\G{}{h}}u)\cdot(\nabla_{\G{}{h}}v)\, ds+\stab\int_{\O(\G{n}{h})}(\bn_h^n\cdot\nabla u)(\bn_h^n\cdot\nabla v) d\,\bx
  \end{split}
\end{equation}
for $u,v\in H^1(\O(\G{n}{h}))$.
We estimate $a_n(v_h,v_h)$ from below,
\begin{equation}\label{lower}
\begin{split}
a_n(v_h,v_h)&=
\nu\|\nabla_{\G{}{h}}v_h\|^2_{\G{n}{h}}+((\Div_{\G{}{h}} (\bw^e-\frac12\wte) v_h,v_h)_{\G{n}{h}} + \stab\|\bn_h^n \cdot \nabla v_h\|_{\O(\G{n}{h})}^2
\\&\ge
\nu\|\nabla_{\G{}{h}}v_h\|^2_{\G{n}{h}} - \xi_h \|v_h\|^2_{\G{n}{h}} +\stab \|\bn_h^n \cdot \nabla v_h\|_{\O(\G{n}{h})}^2
\end{split}
\end{equation}
for any $v_h\in V^n_h$.
 Using \eqref{lower} and \eqref{cond2}  we  check that the bilinear form on the left-hand side of \eqref{e:traceFEM1} is positive definite,
\begin{equation}\label{coer}
\begin{split}
\int_{\G{n}{h}}\frac1{\Delta t}v_h^2\,ds+a_n(v_h,v_h)&\ge
\frac1{2\Delta t}\|v_h\|^2_{\G{n}{h}}+\nu\|\nabla_{\G{}{h}}v_h\|^2_{\G{n}{h}}+\stab\|\bn_h^n \cdot \nabla v_h\|_{\O(\G{n}{h})}^2.
\end{split}
\end{equation}
Hence, due to the Lax-Milgram lemma, the problem in each time step of \eqref{e:traceFEM1} is well-posed.

\medskip
We next derive an \textit{a priori} estimate for the finite element solution to \eqref{e:traceFEM1}.

\begin{theorem}\label{Th1}
  Assume conditions
  \eqref{eq:conddtkappa},
  \eqref{cond1},
  \eqref{cond4},
  \eqref{cond2} and
  \eqref{cond3}, then the solution of \eqref{e:traceFEM1} satisfies the following estimate for $\Delta t$ sufficiently small:
  \begin{equation}\label{FE_stab1}
    \|u_h^n\|^2_{\G{n}{h}} \!+\!\!\sum_{k=1}^{n}{\Delta t}\!\left(\!2 \nu\|\nabla_{\G{}{h}} \!\!u_h^k\|^2_{\G{0}{h}}
\!\! +\!\rho_k\|\bn_h^k\!\! \cdot \! \nabla u_h^k\|_{\O(\G{k}{h})}^2\!\right)\!\le\!\frac32 \exp(c_{\ref{Th1}}t_n)\!\left(\!\|u_h^0\|^2_{\G{0}{h}}\!+\!\tilde\rho_0\|\bn_h^0 \!\!\cdot\! \nabla u_h^0\|_{\O(\G{0}{h})}^2\!\right)\!,
\end{equation}
with $c_{\ref{Th1}}$ independent of $h$, $\Delta t$ and $n$, $\tilde\rho_0= c_{\ref{lem2}b}  \delta_0(\delta_0+h)^{-1}$.
\end{theorem}
\begin{proof}
  We test \eqref{e:traceFEM1} with $v_h=u_h^n$. This leads us to the identity
\[
\frac1{2\Delta t}(\|u_h^n\|^2_{\G{n}{h}}+\|u_h^n-u_h^{n-1}\|^2_{\G{n}{h}})+a_n(u_h^n,u_h^n)=\frac1{2\Delta t}\|u_h^{n-1}\|^2_{\G{n}{h}}.
\]
We drop out the second term, use the lower bound \eqref{lower} and apply \eqref{fund2simple} and assumption \eqref{ass2}
\begin{align}
\|u_h^n& \|^2_{\G{n}{h}}+{2\Delta t}\nu\|\nabla_{\G{}{h}}u_h^n\|^2_{\G{n}{h}}
  +2{\Delta t}\stab\|\bn_h^n \cdot \nabla u_h^n\|_{\O(\G{n}{h})}^2 \\
  & \le\|u_h^{n-1}\|^2_{\G{n}{h}}+ 2 \xi_h \Delta t\|u_h^n\|^2_{\G{n}{h}}  \nonumber \\
  & \le (1 + c_{\ref{lem2}a} \Delta t \|u_h^{n-1}\|^2_{\G{n-1}{h}}) + \Delta t \stabold  \|\bn_h^{n-1} \cdot \nabla u_h^{n-1} \|_{\O(\G{n-1}{h})}^2 + 2 C_0 \Delta t\|u_h^n\|^2_{\G{n}{h}} \nonumber
\end{align}
where the constants $c_{\ref{lem2}a}$ and $C_0$ are independent of $h$, $\Delta t$ and $n$.
We define $c^\ast = c_{\ref{lem2}a} + 2 C_0$ and sum up the inequalities for $n=1,\dots,k$ to get
\begin{multline*}
(1-\Delta t  c^\ast) \|u_h^k\|^2_{\G{k}{h}}+\Delta t\sum_{n=1}^{k}\left(2\nu\|\nabla_{\G{}{h}}u_h^n\|^2_{\G{n}{h}}
+\stab\|\bn_h^n \cdot \nabla u_h^n\|_{\O(\G{n}{h})}^2\right)\\
\le \|u_h^{0}\|^2_{\G{n-1}{}} +\tilde\rho_0\|\bn_h^0 \cdot \nabla u_h^0\|_{\O(\G{0}{h})}^2
+ \Delta t \sum_{n=0}^{k-1} c^\ast  \|u_h^n\|^2_{\G{n}{h}}.
\end{multline*}
Finally, we apply the discrete Gronwall inequality with $\Delta t \leq (2 c^\ast)^{-1}$ to get \eqref{FE_stab1} with $c_{\ref{Th1}} = 2 c^\ast$.
\end{proof}

Now we are ready to devise an error estimate in the energy norm. The proof of the error estimate combines the arguments we used for stability analysis in section~\ref{s_stab} with geometric  and interpolation
error estimates. The geometric  and interpolation error estimates are treated at each time instances $t_n$ for `stationary' surfaces $\G{n}{h}$ and so the developed analysis (cf.~\cite{reusken2015analysis,olshanskii2016trace}) is of help. We start with consistency estimate for \eqref{e:traceFEM1}.

\subsection{Consistency estimate} \label{s:consistency}
While stability analysis dictated us the lower bound \eqref{cond3} for $\stab$, we shall see that the consistency and error analysis leads to a similar natural \emph{upper} bound:
\begin{equation}\label{cond5}\asslabel
 \stab \lesssim (h+\delta_n)^{-1}.
\end{equation}
We assume \eqref{cond5} for the rest of section~\ref{s:Analysis}. Furthermore, in the consistency and error bounds we shall need estimates on derivatives of the solution $u$ in the strip $\O(\Gs)$. By differentiating the identity $u(\bx,t)=u(\bp(\bx),t)$, $(\bx,t)\in\O(\Gs)$, $k\ge 0$ times one finds  that  for $C^{k+1}$-smooth  manifold $\Gs$ the following bound holds:
\begin{equation}\label{u_bound_a}
\|u\|_{W^{k,\infty}(\O(\Gs))}\lesssim \|u\|_{W^{k,\infty}(\Gs)}.
\end{equation}
With a little bit more calculations, see, for example,  \cite[Lemma 3.1]{reusken2015analysis}, one also finds \begin{equation}\label{u_bound_b}
 \|u\|_{H^{k}(U_\ep(\Gamma(t)))}\lesssim \ep^{\frac12}\|u\|_{H^{k}(\Gamma(t))}
\end{equation}
for $t\in[0,T]$ and any such $\ep>0$ that $U_\ep(\Gamma(t))\subset \O(\Gamma(t))$, where $U_\ep(\Gamma(t))$ is the $\ep$-neighborhood in $\mathbb{R}^3$.

We next observe that the smooth solution $u^n=u(t_n)$ of \eqref{transport1} satisfy the identities
\begin{equation}\label{e:consist}
\int_{\G{n}{h}}\left(\frac{u^n-u^{n-1}}{\Delta t} \right) v_h\,ds+ a_n(u^n,v_h) = \consist(v_h),\quad \forall~v_h\in V^n_h,
\end{equation}
with $a_n(\cdot,\cdot)$ as in \eqref{e:def:an} and $\consist(v_h)$ collecting consistency terms due to geometric errors and time derivative approximation, i.e.
\begin{equation*}
\begin{split}
\consist(v_h)&=
\underset{I_1}{\underbrace{\int_{\G{n}{h}}\left(\frac{u^n-u^{n-1}}{\Delta t}\right) v_h\, ds_h-\int_{\G{n}{}}u_t(t_n) v_h^\ell\, ds}}
+
\underset{I_2}{\underbrace{\stab\int_{\O(\G{n}{h})}((\bn_h^n-\bn^n)\cdot\nabla u^n)(\bn_h^n\cdot\nabla v_h) d\bx}}\\
&\quad
+\underset{I_{3,a}}{\underbrace{
    \frac12 \int_{\G{n}{h}} \wte\cdot\nabla_{\G{}{h}} u^n v_h - \wte\cdot\nabla_{\G{}{h}} v_h u^n\, ds_h
    - \frac12 \int_{\G{n}{}} \bw \cdot \nabla u^n v_h^\ell - \bw \cdot \nabla v_h^\ell u^n \, ds}}\\
&\quad
+\underset{I_{3,b}}{\underbrace{\int_{\G{n}{h}} \Div_{\G{}{h}} (\bw^e-\frac12 \wte) u^n v_h\, ds_h-\int_{\G{n}{}}\ \Div_{\Gamma} (\bw - \frac12 \wt) u^n v_h^\ell\, ds}}\\
&\quad
+\underset{I_4}{\underbrace{\nu\int_{\G{n}{h}}\nabla_{\G{}{h}}u^n\cdot\nabla_{\G{}{h}} v_h\, ds_h - \nu\int_{\G{n}{}}\nabla_{\Gamma}u^n\cdot\nabla_{\Gamma} v_h^\ell\, ds}}.
\end{split}
\end{equation*}
We give the estimate for consistency terms in the following lemma.

\begin{lemma}\label{l_consist} Assume $u\in W^{2,\infty}(\Gs)$,
then consistency error has the bound
\begin{equation}\label{est:consist}
|\consist(v_h)|\lesssim (\Delta t+h^q) \norm{u}_{W^{2,\infty}(\Gs)}\left(\|v_h\|_{\G{n}{h}}+\nu^{\frac12}\|\nabla_{\Gamma} v_h\|_{\G{n}{h}} +\stab^{\frac12}\|(\bn_h^n\cdot\nabla v_h)\|_{\O(\G{n}{h})}\right).
\end{equation}
\end{lemma}
\begin{proof} We treat $\consist(v_h)$ term by term, starting with $I_1$:
\[
I_1= - \int_{\G{n}{h}}\int_{t_{n-1}}^{t_n}\frac{t-t_{n-1}}{\Delta t} u_{tt} \,dt\,v_h\, ds+\int_{\G{n}{h}}u_t(t_n) v_h\, ds-\int_{\G{n}{}}u_t(t_n) v_h^\ell\, ds.
\]
We have
\[
\left| - \int_{\G{n}{h}}\int_{t_{n-1}}^{t_n}u_{tt} \frac{t-t_{n-1}}{\Delta t}\,dt\,v_h\, ds\right|\le \frac12 \Delta t\|u_{tt}\|_{\infty,\O(\Gs)}\|v_h\|_{L^1(\G{n}{h})}\lesssim \Delta t\|u\|_{W^{2,\infty}(\Gs)}\|v_h\|_{\G{n}{h}},
\]
and using $u^e_t(\bx,t)=u_t(\bp(\bx),t)-\phi(\bx,t)\bn_t\cdot\nabla_{\Gamma} u(\bp(\bx),t)$ (cf. (6.8) in \cite{olshanskii2014error}),
\begin{align*}
  \int_{\G{n}{h}}u_t(t_n) v_h\, ds_h-\int_{\G{n}{}}u_t(t_n) v_h^\ell\, ds
  & = \int_{\G{n}{h}}\left( (u_t \circ \bp) (1-\mu_h) - \phi \bn_t \cdot \nabla_{\Gamma} u \circ \bp \right)v_h\, ds_h\\
  & \lesssim h^{q+1}(\|u_t\|_{\G{n}{}}+\|\nabla_{\Gamma} u\|_{\G{n}{}})\|v_h\|_{\G{n}{h}},
\end{align*}
where we used \eqref{eq:dist}, $\|\phi\|_{\infty,\G{n}{h}} \lesssim h^{q+1}$, and $\mu_h^n(\bx)ds_h(\bx) =ds(\bp(\bx))$,  $\bx\in\G{n}{h}$, with $\|1-\mu_h\|_{\infty,\G{n}{h}}\lesssim h^{q+1}$; see, e.g., \cite{reusken2015analysis}.
We now turn to estimating the second term,
\begin{equation*}
\begin{split}
  |I_2|&\le \stab\|(\bn_h^n-\bn^n)\cdot\nabla u^n\|_{\O(\G{n}{h})}\|\bn_h^n\cdot\nabla v_h\|_{\O(\G{n}{h})}\\&\lesssim \stab h^q\|\nabla u^n\|_{\O(\G{n}{h})}\|\bn_h^n\cdot\nabla v_h\|_{\O(\G{n}{h})}
  \lesssim \stab h^{q} (\delta_n + h)^{\frac12}\|\nabla_\Gamma u^n\|_{\G{n}{h}}\|\bn_h^n\cdot\nabla v_h\|_{\O(\G{n}{h})}.
\end{split}
\end{equation*}
In the last inequality we used \eqref{u_bound_b}. Recalling the condition \eqref{cond5} for $\stab$,
we find
$$|I_2|\lesssim  h^{q} \stab^{\frac12}\|\nabla_\Gamma u^n\|_{\G{n}{h}}\|\bn_h^n\cdot\nabla v_h\|_{\O(\G{n}{h})}.$$

The consistency terms $I_{3,a}$, $I_{3,b}$, $I_4$ are standard in TraceFEM on steady surfaces. One has the bounds, see \cite[Lemma 7.4]{gross2015trace} or \cite[Lemma 5.5]{reusken2015analysis},
\[
|I_{3,a}|+|I_4|\lesssim h^{q+1}\left(\|\nabla_\Gamma u^n\|_{\G{n}{}}\|v_h\|_{\G{n}{h}} +\|\nabla_\Gamma u^n\|_{\G{n}{}}\|\nabla_{\G{}{h}}v_h\|_{\G{n}{h}} \right).
\]
We fix  $t=t_n$ and skip the dependence on time in our notation up to the end of the proof.
To handle the term with divergence, introduce orthogonal projectors,
\begin{equation*}
 \bP(\bx):=\bI-\bn^n(\bx)\bn^n(\bx)^T, \quad \hbox{for }\bx\in \O(\Gamma^n),\qquad
 \bP_h(\bx):=\bI-\bn_{h}^n(\bx)\bn_{h}^n(\bx)^T, \quad \hbox{for }\bx\in \G{n}{h}.
\end{equation*}
For the surface divergence one has the following representation:
\begin{equation}\label{e:Marz}
\Div_\Gamma\bw=\operatorname{tr}(\nabla_\Gamma\bw)= \operatorname{tr}(\bP\nabla\bw)~~\text{and}~~
\Div_{\G{}{h}}\bw=\operatorname{tr}(\nabla_{\G{}{h}}\bw)= \operatorname{tr}(\bP_h\nabla\bw).
\end{equation}
Take $\bx\in\G{}{h}$, not lying on an edge. Using  $ \nabla u(\mathbf{x}) = (\mathbf{I}-\phi(\mathbf{x})\mathbf{H})\nabla_{\Gamma} u(\bp(\bx))$,  $\bx\in\O(\G{n}{})$, we obtain
\begin{align*}
 \Div_{\G{}{h}}\bw^e(\bx) & =\operatorname{tr}(\bP_h\nabla\bw^e(\bx))=
\operatorname{tr}\left(\bP_h(\mathbf{I}-\phi(\mathbf{x})\mathbf{H})\nabla_{\Gamma} \bw(\bp(\bx))\right)\\&=
\operatorname{tr}\left(\bP\nabla_{\Gamma} \bw(\bp(\bx))\right)
+\operatorname{tr}\left((\bP_h-\bP)\nabla_{\Gamma} \bw(\bp(\bx))\right)-
\phi(\mathbf{x})\operatorname{tr}\left(\bP_h\mathbf{H}\nabla_{\Gamma} \bw(\bp(\bx))\right)\\&=
\Div_\Gamma\bw(\bp(\bx))
+\operatorname{tr}\left((\bP_h-\bP)\nabla_{\Gamma} \bw(\bp(\bx))\right)-
\phi(\mathbf{x})\operatorname{tr}\left(\bP_h\mathbf{H}\nabla_{\Gamma} \bw(\bp(\bx))\right).
\end{align*}
Thanks to \eqref{eq:dist} and \eqref{eq:normals}, we bound the last two terms at the right-hand side
\[
|\bP_h-\bP|\lesssim h^q,\quad |\phi(\mathbf{x})\bP_h\mathbf{H}|\lesssim h^{q+1}.
\]
We proved the estimate $|\Div_{\G{}{h}}\bw^e-\Div_\Gamma(\bw\circ\bp)|\lesssim h^{q}$ on $\G{n}{h}$.
With the help of this estimate and the similar one  with $\bw$ replaced by $\wt$, we bound $I_{3,b}$ term,
\begin{equation*}
|I_{3,b}|=\left|\int_{\G{n}{h}} \left(\Div_{\G{}{h}} (\bw^e-\frac12 \wte)- \mu_h \Div_{\Gamma} (\bw - \frac12 \wt)\circ\bp\right) u^n v_h\, ds_h\right|\ \lesssim h^q \|u^n\|_{\G{n}{}}\|v_h\|_{\G{n}{h}}.
\end{equation*}
\end{proof}

\begin{remark}\rm \label{rem_cons} The $h$-dependence of the consistency estimate in \eqref{est:consist} is due to the geometric errors. Increasing the accuracy of the surface recovery leads to better consistency in \eqref{est:consist}.
The order of the estimate can be improved with respect to $h$ if more  information about $\Gamma$ is available.
For example, if one can use $(\Div_\Gamma(\bw-\frac12 \wt))^e$ instead of  $\Div_{\G{}{h}}(\bw^e-\frac12 \wte)$ on $\G{}{h}$, then the $O(h^q)$ term on the right-hand side of \eqref{est:consist} is replaced by $O(h^{q+1})$.
\end{remark}

\subsection{Error estimate in the energy norm} \label{sec:aprioriest}
Denote the error function $\err^n=u^n-u^n_h$, $\err^n\in H^1(\O(\G{n}{h}))$. From \eqref{e:traceFEM1} and \eqref{e:consist} we get the error equation,
\begin{equation}\label{e:err}
\int_{\G{n}{h}}\left(\frac{\err^n-\err^{n-1}}{\Delta t} \right) v_h\,ds+ a_n(\err^n,v_h) = \consist(v_h),\quad \forall~v_h\in V^n_h.
\end{equation}
 We let $u_I^n\in V_h^n$ be an interpolant  for $u^n$ in $\O(\G{n}{h})$; we assume $u^n$  sufficiently smooth so that the interpolation is well-defined. Following standard lines of argument, we split $\err^n$ into finite element and approximation parts,
\[
\err^n=\underset{\mbox{$e^n$}}{\underbrace{(u^n-u^n_I)}}+\underset{\mbox{$e^n_h$}}{\underbrace{(u^n_I-u^n_h)}}.
\]
Equation \eqref{e:err} yields
\begin{equation}\label{e:err1}
\int_{\G{n}{h}}\left(\frac{e^n_h-e^{n-1}_h}{\Delta t} \right) v_h\,ds+ a_n(e_h^n,v_h) = \interpol(v_h)+\consist(v_h),\quad \forall~v_h\in V^n_h,
\end{equation}
with the interpolation term
\[
\interpol(v_h)=-\int_{\G{n}{h}}\left(\frac{e^n-e^{n-1}}{\Delta t} \right) v_h\,ds_h- a_n(e^n,v_h).
\]
We give the estimate for interpolation terms in the following lemma.

\begin{lemma}\label{l_interp} Assume $u\in W^{m+1,\infty}(\Gs)$ and $\Gs$ is sufficiently smooth,
then it holds
\begin{equation}\label{est_inter}
  |\interpol(v_h)|\lesssim h^m \, \norm{u}_{W^{m+1,\infty}} \, (\|v_h\|_{\G{n}{h}}+\nu^{\frac12}\|\nabla_{\Gamma_h} v_h\|_{\G{n}{h}}).
\end{equation}
\end{lemma}
\begin{proof}
We need Hansbo's trace inequality~\cite{Hansbo02},
\begin{equation}\label{trace}
\|v\|_{S\cap\G{n}{h}}\le c(h^{-\frac12}\|v\|_{S}+h^{\frac12}\|\nabla v\|_{S}),\quad ~~v\in H^1(S),~~S\in \mathcal{T}_h^{\Gamma},
\end{equation}
with some $c$ independent of $v$, $T$, $h$, $\G{n}{h}$.
Under mild assumptions on the resolution of the smooth surface $\G{n}{h}$ by the mesh (cf. \cite[Assumption 4.1(A2)]{reusken2015analysis}) the inequality has been proven in \cite[Lemma 4.3]{reusken2015analysis}.
We use interpolation properties of polynomials and their traces. In particular,
\begin{equation}\label{eq:interp}
\min_{v_h\in V^h}\left(\|v^e-v_h\|_{\G{n}{h}}+h\|\nabla(v^e-v_h)\|_{\G{n}{h}}\right)\lesssim h^{m+1}\|v\|_{H^{m+1}(\G{n}{})}\quad \text{for}~v\in H^{m+1}(\G{n}{});
\end{equation}
see, e.g., \cite{gross2015trace,reusken2015analysis,olshanskii2016trace}.
With the help of \eqref{trace} we treat the first term in $\interpol(v_h)$,
\begin{equation}\label{e:err2}
\begin{split}
\left|\int_{\G{n}{h}}\right.&\left.\left(\frac{e^n-e^{n-1}}{\Delta t} \right) v_h\,ds_h\right|\le \left\|\frac{e^n-e^{n-1}}{\Delta t}\right\|_{\G{n}{h}}\|v_h\|_{\G{n}{h}}\\
&\lesssim \left(h^{-\frac12} \Delta t^{-1} \| e^n-e^{n-1}\|_{\O_\Gamma(\G{n}{h})}+h^{\frac12} \Delta t^{-1} \|\nabla (e^n-e^{n-1})\|_{\O_\Gamma(\G{n}{h})}\right)\|v_h\|_{\G{n}{h}}.
\end{split}
\end{equation}
Now, using condition \eqref{cond1} we handle the first term on the right-hand side of \eqref{e:err2},
\begin{align*}
  \| e^n & -e^{n-1} \|^{2}_{\O_\Gamma(\G{n}{h})}
  =\| e(t_n)-e(t_{n-1}) \|_{\O_\Gamma(\G{n}{h})}^2
    = \left\|\int^{t_n}_{t_{n-1}} e_t(t')\,dt'\right\|_{\O_\Gamma(\G{n}{h})}^2
  &
  \\
  &
    \le \Delta t\int^{t_n}_{t_{n-1}}\| e_t(t')\|_{\O_\Gamma(\G{n}{h})}^2\,dt'
    \lesssim |\Delta t|^2\,h^{2m}\sup_{t\in[t_{n-1},t_n]} \| u_t\|_{H^{m}(\O_\Gamma(\G{n}{h}))}^2
  & ({\footnotesize \text{Cauchy-Schwarz and \eqref{eq:interp}}})
  \\
  &
    \lesssim |\Delta t|^2\,h^{2m+1}\sup_{t\in[t_{n-1},t_n]} \| u_t\|_{H^{m}(\G{n}{h})}^2
    \lesssim |\Delta t|^2\,h^{2m+1}\| u\|_{W^{m+1,\infty}(\Gs)}^2.
  &  ({\footnotesize \text{by~\eqref{u_bound_b}}})
\end{align*}
We estimate the second term on the right-hand side of \eqref{e:err2}, using similar arguments,
\[
\|\nabla(e^n-e^{n-1})\|_{\O_\Gamma(\G{n}{h})}\lesssim |\Delta t|^2\,h^{2m-1}\| u\|_{W^{m+1,\infty}(\Gs)}^2.
\]
We handle the term $a_n(e^n,v_h)$ using the Cauchy-Schwarz inequality and interpolation properties of $v_h$ in the straight-forward way. This leads to the estimate
\[
|a_n(e^n,v_h)|\lesssim h^m\|u\|_{H^{m+1}(\G{n}{})}(\|v_h\|_{\G{n}{h}}+\nu^{\frac12}\|\nabla_{\Gamma_h} v_h\|_{\G{n}{h}}).
\]
We summarize the above bounds into the estimate of the interpolation term as in \eqref{est_inter}.
\end{proof}

Now we are prepared to prove the main result of the paper. Let $u_h^0=u_I^0\in V_h^0$ be a suitable interpolant to
$u^0\in \O(\Gamma_h^0)$.


\begin{theorem}\label{Th2}
  Assume \eqref{phi_h}--\eqref{phi_hbb}, \eqref{eq:conddtkappa}--\eqref{cond1},
  \eqref{cond4},
  \eqref{cond2},
  \eqref{cond3}, and
  \eqref{cond5},  and $\Delta t$ is sufficiently small,  $u$ is the  solution to \eqref{transport},  $u\in W^{m+1,\infty}(\Gs)$, $m\geq 1$, $\Gs$ is sufficiently smooth.  
  For  $u_h^n$, $n=1,\dots,N$, the finite element solution of \eqref{e:traceFEM1},
  and $\err^n = u_h^n - u^n$  the following error estimate holds:
  \begin{multline}\label{FE_est1}
    \|\err^n\|^2_{\G{n}{h}}+{\Delta t}\sum_{k=1}^{n}\! \left( \nu\|\nabla_{\G{}{h}} \err^k \|^2_{\G{k}{h}}
      +\stab\|\bn_h^k \cdot \nabla\err_h^k\|_{\O(\G{k}{h})}^2\right)
    \lesssim \exp(c_{\ref{Th2}}t_n) R(u) (\Delta t^2+h^{2\min\{m,q\}}),
  \end{multline}
  with $R(u) := \norm{u}_{W^{m+1,\infty}(\Gs)}^2$ and $c_{\ref{Th2}}$ independent of $h$, $\Delta t$, $n$ and of the position of the surface over the background mesh.
\end{theorem}
\begin{proof}
The arguments largely repeat those used to show the stability result in  Theorem~\ref{Th1} and involve estimates from Lemmas~\ref{l_consist} and \ref{l_interp} to bound the  arising right-hand side terms.
We set $v_h=2\Delta t e^n_h$ in \eqref{e:err1}. This gives
\[
\|e_h^n\|^2_{\G{n}{h}}- \|e_h^{n-1}\|^2_{\G{n}{h}} +\|e_h^n-e_h^{n-1}\|^2_{\G{n}{h}}+{2\Delta t} a_n(e_h^n,e_h^n)=2\Delta t(\interpol(e_h^n)+\consist(e_h^n))
\]
Dropping the third term, using the lower bound \eqref{lower} for $a_n$ and estimating  $ \|e_h^{n-1}\|^2_{\G{n}{h}}$  with \eqref{fund2simple} yields
\begin{multline*}
\|e_h^n\|^2_{\G{n}{h}}+{2\Delta t}\nu\|\nabla_{\G{}{h}}e_h^n\|^2_{\G{n}{h}}
+2{\Delta t}{\stab}\|\bn_h^n \cdot \nabla e_h^n\|_{\O(\G{n}{h})}^2\\
\le(1+c_{\ref{lem2}}^\ast \Delta t) \|e_h^{n-1}\|^2_{\G{n-1}{h}}
+ \Delta t \stabold  \|\bn_h^{n-1} \cdot \nabla e_h^{n-1}\|_{\O(\G{n-1}{h})}^2
+ 2\xi_h \Delta t \|e_h^n\|^2_{\G{n}{h}} + 2\Delta t(\interpol(e_h^n)+\consist(e_h^n)).
\end{multline*}
We recall assumption \eqref{ass2} and the defintion $c^\ast = c_{\ref{lem2}}^\ast + 2 C_0$
(cf. the proof of  Theorem~\ref{Th1})
to obtain
\begin{multline}\label{aux8}
(1- c^\ast \Delta t)\|e_h^n\|^2_{\G{n}{h}}+{2\Delta t}\nu\|\nabla_{\G{}{h}}e_h^n\|^2_{\G{n}{h}}
+2{\Delta t}{\stab}\|\bn_h^n \cdot \nabla e_h^n\|_{\O(\G{n}{h})}^2\\
\le(1+ c^\ast \Delta t)\|e_h^{n-1}\|^2_{\G{n-1}{h}}
+\Delta t\stabold \|\bn_h^{n-1} \cdot \nabla e_h^{n-1}\|_{\O(\G{n-1}{h})}^2
+ 2\Delta t(\interpol(e_h^n)+\consist(e_h^n)).
\end{multline}
To estimate the interpolation and consistency terms, we apply Young's inequality to the right-hand sides of \eqref{est:consist} and \eqref{est_inter} yielding
\begin{align*}
  2 \Delta t\consist(e_h)
  &\le c\,\Delta t (\Delta t^2+h^{2q}) \norm{u}_{W^{2,\infty}(\Gs)}^2
  &&
    +\frac{\Delta t}{2}\left(\|e_h^n\|_{\G{n}{h}}^2+\nu\|\nabla_{\Gamma_h} e_h^n\|_{\G{n}{h}}^2 +\stab\|(\bn_h^n\cdot\nabla e_h^n)\|_{\O(\G{n}{h})}^2\right),\\
  2 \Delta t\interpol(e_h)
  & \le c\,\Delta t ~ h^{2m} \norm{u}_{W^{m+1,\infty}(\Gs)}^2
  &&
    +\frac{\Delta t}{2}\left(\|e_h^n\|_{\G{n}{h}}^2+\nu\|\nabla_{\Gamma_h} e_h^n\|_{\G{n}{h}}^2\right),
\end{align*}
with a constant $c$ independent of $h$, $\Delta t$, $n$ and of the position of the surface over the background mesh. Substituting this in \eqref{aux8} and
summing up the resulting inequalities for $n=1,\dots,k$ and noting $e^0_h=0$ in $\O(\Gamma_h^0)$ we get
\begin{multline*}
(1- (c^\ast+1) \Delta t  )\|e_h^k\|^2_{\G{k}{h}}+\Delta t\sum_{n=1}^{k}\left(\nu\|\nabla_{\G{}{h}}e_h^n\|^2_{\G{n}{h}}
+\stab\|\bn_h^n \cdot \nabla e_h^n\|_{\O(\G{n}{h})}^2\right)\\
\le 
\Delta t \sum_{n=0}^{k-1} c^\ast  \|e_h^n\|^2_{\G{n}{h}} + c\,\norm{u}_{W^{m+1,\infty}}^2 (\Delta t^2+h^{2q}+h^{2m}).
\end{multline*}
We apply the discrete Gronwall inequality with $\Delta t \leq (2+2c^\ast)^{-1}$ to get
\begin{align}
    \|e_h^k & \|^2_{\G{k}{h}}+\sum_{n=1}^{k}{\Delta t} \left( \nu\|\nabla_{\G{}{h}}e_h^n\|^2_{\G{n}{h}}
              +\stab\|\bn_h^n \cdot \nabla e_h^n\|_{\O(\G{n}{h})}^2\right) \\
              & \lesssim \exp(c_{\ref{Th2}} t_k)
                \norm{u}_{W^{m+1,\infty}(\Gs)}^2 (\Delta t^2+h^{2\min\{m,q\}}) =: \exp(c_{\ref{Th2}} t_k) Q_{e},
                \nonumber
\end{align}
with $c_{\ref{Th2}} = 2 (c^\ast+1)$.
The triangle inequality, standard FE interpolation properties, \eqref{eq:interp} and \eqref{cond5} give
\begin{equation*}
\begin{split}
    \|\err^k\|^2_{\G{k}{h}}&+\sum_{n=1}^{k}{\Delta t} \left( \nu\|\nabla_{\G{}{h}}\err^n\|^2_{\G{n}{h}}
      +\stab\|\bn_h^n \cdot \nabla \err^n\|_{\O(\G{n}{h})}^2\right)
    \\ &
    \leq
    Q_{e} +\|e^k\|^2_{\G{k}{h}}+\sum_{n=1}^{k}{\Delta t} \left( \nu\|\nabla_{\G{}{h}}e^n\|^2_{\G{n}{h}}
      +\stab\|\bn_h^n \cdot \nabla e^n\|_{\O(\G{n}{h})}^2\right)
    \\ &
    \lesssim Q_{e}+ \norm{u}_{H^{m+1}(\G{k}{})} h^{2m} \underbrace{( 1 +\stab(\delta_n+h))}_{\lesssim 1}.
\end{split}
\end{equation*}
  This completes the proof.
\end{proof}

\begin{remark}[Extension of the analysis to BDF2] \label{rem:bdf2} \rm
The method is extendable to higher order time stepping methods. To keep the analysis manageable, we restricted to the backward Euler discretization. Here, we briefly summarize what needs to be considered for an extension of the analysis to higher order schemes. We consider the BDF2 schemes here. Obviously the finite difference stencil for the time derivative is changed from $\frac{u^n-u^{n-1}}{\Delta t}$ to $\frac{3u^n-4u^{n-1}+u^{n-2}}{2\Delta t} $ in the semi-discrete method in
\eqref{e:ImEuler},\eqref{e:conttraceFEM1} and for the fully discrete method in \eqref{e:traceFEM1}.
Accordingly, the layer width of the extension has be increased so that $\G{n}{} \subset \mathcal{O}(\G{n-1}{}) \cap \mathcal{O}(\G{n-2}{})$ and $\G{n}{h} \subset \mathcal{O}(\G{n-1}{h}) \cap \mathcal{O}(\G{n-1}{h})$. To this end, we have to change $\delta_n$ in \eqref{e:delta} to $\delta_n = 2 c_\delta \sup_{t\in[t_{n-2},t_n]} \Vert \wn \Vert_{L^{\infty}(\Gamma(t))} \Delta t$. Further, in the proof of the coercivity in the (spatially) continuous and discrete setting we have to change the time step restrictions \eqref{e:xi} and \eqref{cond2} according to the changed coefficient in the BDF formula. The Gronwall-type arguments in section \ref{s:stab:semi-disc} and in Theorem \ref{Th1} have to be replaced with corresponding versions for the BDF scheme. To handle the time derivative terms, a special
norm should be used~\cite{hairer1996solving}, which is a linear combination of $L^2$ surface norm at $n$ and $n-1$ time steps. Finally, the consistency analysis in section \ref{s:consistency} can then be improved, specifically the term $I_1$ leading to a higher order (in $\Delta t$) estimate in Lemma \ref{l_consist} and Theorem \ref{Th2}.
\end{remark}

\section{Algebraic stability} \label{s:algebra}
In every time step we have to solve a linear system of the form
$$
  \bA \bx = \bbf \quad \text{ with } \bA \in \rr^{N\times N} , \bbf, \bx \in  \rr^{N},
$$
where $N = \operatorname{dim}(V_h^n)$, $\bA$ and $\bbf$ are the matrix and vector corresponding to the involved bilinear form and the right-hand side linear form, whereas $\bx$ is the solution vector.
We split the left-hand side bilinear form into its symmetric and skew-symmetric part and define
\begin{subequations}
\begin{align}
  A_n(u,v) := & B_n(u,v) + C_n(u,v) \quad \big(= \int_{\G{n}{h}} \frac{1}{\Delta t} u v ~ ds + a_n(u,v)\big), && u,v \in V_h^n, \\
  B_n(u,v) := & \int_{\G{n}{h}} (\frac{1}{\Delta t} + \Div_{\G{}{h}} (\bw^e - \frac12 \wte)) u v ~ ds
 \\ & +\nu\int_{\G{n}{h}}(\nabla_{\G{}{h}}u)\cdot(\nabla_{\G{}{h}}v)\, ds+\stab\int_{\O(\G{n}{h})}(\bn_h^n\cdot\nabla u)(\bn_h^n\cdot\nabla v) d\bx, && u,v \in V_h^n, \nonumber
  \\
  C_n(u,v) := & \int_{\G{n}{h}} \frac12 (\wte\cdot\nabla_{\G{}{h}} u) v - \frac12 (\wte\cdot\nabla_{\G{}{h}} v) u ~ ds, && u,v \in V_h^n.
\end{align}
\end{subequations}
Correspondingly we denote by $\bB$ and $\bC \in \rr^{N\times N}$ the matrices to the bilinear forms $B_n$ and $C_n$.

To bound the spectral condition number of $\bA$, we use the following result~ \cite[Theorem 1]{ElmanSchultz86}:
\begin{lemma}\label{lem:elman}
  With $\bA \in \rr^{N \times N}$, $\bB = \frac12( \bA + \bA^T)$ and $\bC = \frac12( \bA - \bA^T)$, for the spectral condition number of $\bA$ there holds
  \begin{equation}
    \kappa(\bA) \leq \frac{\lambda_{\max}(\bB) + \rho(\bC)}{\lambda_{\min}(\bB)}
  \end{equation}
  where $\rho(\cdot)$ denotes the spectral radius and $\lambda_{\max}(\bB)$ and $\lambda_{\min}(\bB)$ are the largest and smallest eigenvalues of the symmetric and positive definite matrix $\bB$.
\end{lemma}
To estimate $\kappa(\bA)$, we derive bounds for $\rho(\bC)$, $\lambda_{\min}(\bB)$ and $\lambda_{\max}(\bB)$ in the next two lemmas.
\begin{lemma} \label{lem:rhom}
  There holds
    $\rho(\bC) \lesssim \Vert \bw \Vert_{\infty} h^{d-2}$.
\end{lemma}
\begin{proof}
  We note that $\bC$ is skew-symmetric and hence a normal matrix. Thus, we have
  \begin{equation}
    \rho(\bC) = \max_{\bx \in \cc^{N}} \frac{\bx^T \bC \bx}{\bx^T \bx}.
  \end{equation}
  Now let $\bx \in \cc^n$ and $v$ be the corresponding finite element function in $V_h + i V_h$ (where $i$ is the imaginary unit), then we have
  \begin{equation} \label{eq:xCxest}
\hspace*{-0.1cm}    \bx^T \bC \bx = C_n(v,v) \leq  \Vert \bw \Vert_{\infty}  \Vert \nabla_{\G{}{}} v \Vert_{\G{}{h}} \Vert v \Vert_{\G{}{h}}
    \lesssim \Vert \bw \Vert_{\infty} h^{-2} \Vert v \Vert_{\O(\G{n}{h})}
    \simeq
    \Vert \bw \Vert_{\infty} h^{1} \Vert \bx \Vert_2^2,
  \end{equation}
  where we made use of inverse inequalities and $\Vert v \Vert_{\O(\G{n}{h})} \simeq h^3 \Vert \bx \Vert_2^2$.
\end{proof}

\begin{lemma} \label{lem:lambdas}
  Under conditions \eqref{eq:conddtkappa} and \eqref{cond2}, there holds
  \begin{subequations}
  \begin{align} \label{eq:lambdamax}
    \lambda_{\max}(\bB)
    & \lesssim
    h^{d-2} \left( \frac{h}{\Delta t} + \frac{\nu}{h} + \stab \right), \\
    \lambda_{\min}(\bB) \label{eq:lambdamin}
    & \gtrsim h^3 \left(( \delta_n + h) \left(\Delta t + \frac{\delta_n + h}{\stab}\right) \right)^{-1}.
  \end{align}
  \end{subequations}
\end{lemma}
\begin{proof} Estimate
  \eqref{eq:lambdamax} follows with \eqref{cond2}, $\Delta t < (4 \xi_h)^{-1}$, standard FE inverse and trace inequalities similar to \eqref{eq:xCxest}.
  Then again, with \eqref{coer} and Theorem \ref{lemcrucial} we easily obtain \eqref{eq:lambdamin} with
  \begin{align}
h^3 \Vert \bx \Vert_2^2 & \simeq
    \Vert u \Vert_{\O(\G{n}{h})}^2
    \lesssim (\delta_n + h) \Vert u \Vert_{\G{n}{h}}^2 + (\delta_n + h)^2 \Vert \bn_h^n \cdot \nabla u\Vert_{\O(\G{n}{h})}^2 \\
    & \leq (\delta_n + h) \max \Big\{ {\Delta t}, \frac{\delta_n + h}{\stab} \Big\} \cdot B_n(u,u)
    \leq (\delta_n + h) \left( \Delta t + \frac{\delta_n + h}{\stab} \right) \cdot \underbrace{B_n(u,u)}_{= \bx^T \bB \bx}. \nonumber   \end{align}
\end{proof}
\begin{corollary}\label{cor:condest}
The estimates in Lemma \ref{lem:rhom} and Lemma \ref{lem:lambdas}  plugged into Lemma \ref{lem:elman} result in the following condition number bound:
  \begin{equation} \label{eq:condest}
    \kappa(\bA) \lesssim \underbrace{\frac{\delta_n + h}{h}}_{K_{1}}
      \Big(
        \underbrace{\frac{1}{\Delta t} + \frac{\nu}{h^2} + \frac{\Vert \bw \Vert_{\infty}}{h}}_{K_{2,a}}
        +
        \underbrace{\frac{\stab}{h} }_{K_{2,b}} \Big)
      \underbrace{\left( \Delta t + \frac{\delta_n + h}{\stab}\right)}_{K_{3}}
  \end{equation}
\end{corollary}
We notice that in these condition number estimates no assumption on the scaling on the stabilization parameter was used.
\begin{remark}[Discussion of Corollary \ref{cor:condest}] \label{rem:coroll}\rm
Let us discuss the terms on the right hand side of \eqref{eq:condest}.
The first term, $K_1$ describes the layer thickness in terms of elements. We note that this term is bounded by a constant for $\Delta t \lesssim h$. Otherwise the condition  number will increase with an increasing (element) layer thickness.
In the second term, $K_2=K_{2,a}+K_{2,b}$, we first note that the latest contribution $K_{2,b}=\frac{\stab}{h}$ can be absorbed by $\frac{\nu}{h^2}$ if condition \eqref{cond5} is fulfilled and $\nu = \mathcal{O}(1)$.
For the last term, $K_3$, we can use condition \eqref{cond3} to bound $K_3 \lesssim \Delta t + h^2$.
Assume $\Delta t$ is the dominating summand in $K_3$.
Then, there holds $K_2 \cdot K_3 \simeq 1 + \frac{\nu \Delta t}{h^2} + \frac{\norm{\bw}_{\infty} \Delta t}{h}$ which is the usual condition number scaling known from fitted convection diffusion equation discretizations on stationary domains which is the best that we can expect in our setting.
\end{remark}

\section{Numerical experiments}\label{s:Numerics}
In this section, we will show some numerical experiments for the proposed method. The results demonstrate the accuracy of the
stabilized TraceFEM and verify the analysis results on error estimates and condition number bounds.

All implementations are done in the finite element package DROPS \cite{DROPS}.
We applied both the backward Euler scheme and the BDF2 scheme to approximate the time derivative.
At each time step, we assemble the stiffness matrix and the right-hand side by numerical integration
over the discrete surfaces $\G{n}{h}$ which is obtained by piecewise linear interpolation $\phi_h^n$ of the exact level set function $\phi^n$,
$\G{n}{h}=\{\bx\in\mathbb{R}^3~:~\phi_h^n(\bx)=0\}$,
i.e. $q=1$ in \eqref{phi_h}.
For the disretization in space we consider piecewise linears, i.e. $k=1$ in \eqref{eq:Vh}.
The computational domain in all considered examples is $\Omega=[-2,2]^3$ which contains $\G{}{}(t)$ (and $\G{}{h}(t)$) at all times $t \in [0,T]$.
To arrive at a computation mesh, we use a combination of uniform subdivision into cubes with side length $h$ and a Kuhn subdivision into 6 tetrahedra. This results in the shape regular background triangulation $\T_h$.
The temporal grid is chosen uniform in all experiments, $t_n=n\Delta t$ with $\Delta t=\frac{T}{N}$.
For the narrow band zone we choose $c_{\delta}=2.5$ in \eqref{e:delta} which is sufficient for the backward Euler and the BDF2 scheme.
All linear systems are solved using GMRES with a Gauss--Seidel preconditioner to a relative tolerance of $10^{-15}$.

In the experiments we are interested in the $L^2(0,T;H^1(\G{}{h}(t))$ surface norms, which we approximate using the trapezoidal quadrature rule in time, and $L^\infty(0,T;L^2(\G{}{h}(t))$ which we approximate by $\max_{n=1,..,N} \norm{\cdot}_{L^2(\G{n}{h}(t))}$.
To investigate the rates of convergence we apply successive refinements in space and in time.
The numerical results of these convergence studies are supplemented by ``experimental orders of convergence''(\eoc{}) in space and time where \eoc{} $=  \log_2(e_b/e_a)$ for two successive errors $e_a$ and $e_b$. We use a subscript \texttt{x} or \texttt{t} for every refinement in space or time, respectively, that has been applied between the two compared errors. This mean that \eoc{x} / \eoc{t} denote the usual \eoc{} for one refinement in space / time. For two time levels of refinement at once between the comparison, we have \eoc{tt} as in Tables \ref{tab:table1} and \ref{tab:table2}. Consequently, combined refinements in space and time with $h \sim \Delta t$ are denoted \eoc{xt}, whereas combined refinements with $h \sim \Delta t^2$ are denoted by \eoc{xtt}.

For the different test problems, below, we apply the backward Euler scheme and the BDF2 scheme.
In the first experiment we consider two different scalings for $\stab$
\begin{subequations} \label{eq:rhoscales}
\begin{align}
  \stab &\sim 1, \label{eq:rhoscales1}\\
  \stab &\sim \frac{\nu}{\delta_h + h} + \norm{\bw}_{\infty}, \label{eq:rhoscales2}
\end{align}
\end{subequations}
where only the latter scaling fulfills the lower bound \eqref{cond3} for the stability analysis. The scaling in the parameters $\bw$ and $\nu$ is motivated by scaling arguments.
We choose the constants so that $\stab = 4$ and $\stab = \frac{\nu}{\delta_h+h} + \norm{\bw}_{\infty}$ and evaluate errors as well as condition numbers.
In the other experiments we only consider $\stab = \frac{\nu}{\delta_h+h} + \norm{\bw}_{\infty}$.
\begin{table}
\caption{$L^2(H^1)$- and $L^{\infty}(L^2)$-norm error in Experiment~1 with backward Euler and $\stab = 4$.}\label{tab:table1}
\vspace{-0.2cm}
\begin{center}
  \footnotesize
  \begin{tabular}{@{~}l@{~~}r@{~~}r@{~~}r@{~~}r@{~}r@{~}}
    \toprule
    \multicolumn{5}{c}{$L^2(H^1)$-norm of the error} \\
    \midrule
    & $h=1/2$
    & $h=1/4$
    & $h=1/8$
    & $h=1/16$ & \eoc{tt} \\
    \midrule
    $\Delta t=1/8$
    &\numbf{0.93099}
    &\num{0.612822}
    &\num{0.376864}
    &\num{0.244665} & -- \\
    $\Delta t=1/32$
    &\num{0.915289}
    &\numbf{0.63117}
    &\num{0.348896}
    &\num{0.181064} & \numQ{0.434308003335583}\\
    $\Delta t=1/128$
    &\num{0.916754}
    &\num{0.640278}
    &\numbf{0.34972}
    &\num{0.176787} & \numQ{0.03448754104396}\\
    $\Delta t=1/512$
    &\num{0.917529}
    &\num{0.643064}
    &\num{0.350548}
    &\numbf{0.17682} & \numQ{-0.000269275984285}\\
    \midrule
    \multicolumn{1}{r}{\eoc{x}}
    &{---}
    &\numQ{0.512791430602081}
    &\numQ{0.875350323022119}
    &\numQ{0.987330537485532}\\
    \midrule
    \multicolumn{1}{r}{\underline{\eoc{xtt}}}
    &{---}
    &\numQ{0.560737036947621}
    &\numQ{0.851828330449777}
    &\numQ{0.983918837820405}\\
    \bottomrule
  \end{tabular}
  \begin{tabular}{@{~}l@{~~}r@{~~}r@{~~}r@{~~}r@{~}r@{~}}
    \toprule
    \multicolumn{5}{c}{$L^{\infty}(L^2)$-norm of the error} \\
    \midrule
    & $h=1/2$
    &$h=1/4$
    &  $h=1/8$
    & $h=1/16$ & \eoc{tt} \\
    \midrule
    \vphantom{$\Delta t=1/8$}
    &\numbf{0.223854}	
    &\num{0.117571}	
    &\num{0.14301}	
    &\num{0.157484} & --- \\
    \vphantom{$\Delta t=1/32$}
    &\num{0.330099}	
    &\numbf{0.133048}
    &\num{0.0363278}	
    &\num{0.0397675} & \numQ{1.98554348717337}\\
    \vphantom{$\Delta t=1/128$}
    &\num{0.364026}
    &\num{0.159023}	
    &\numbf{0.0298081}
    &\num{0.0121568} &\numQ{1.70982634715732} \\
    \vphantom{$\Delta t=1/512$}
    &\num{0.372923}
    &\num{0.166124}
    &\num{0.033647}
    &\numbf{0.0077501} & \numQ{0.649476691428914}\\
    \midrule
    {---}
    &\numQ{1.16661726289071}
    &\numQ{2.30371073152379}
    &\numQ{2.11819104764506}\\
    \midrule
    {---}
    &\numQ{0.750611273998059}
    &\numQ{2.15817049923177}
    &\numQ{1.94341758854398}\\
    \bottomrule
  \end{tabular}
\end{center}
\vspace*{-0.25cm}
\end{table}
\begin{table}\small
\caption{$L^2(H^1)$- and $L^{\infty}(L^2)$-norm error in Experiment~1 with backward Euler and $\stab = \norm{\bw}_{\infty} + \nu (\delta_h+h)^{-1}$.}\label{tab:table2}
\vspace{-0.2cm}
\begin{center}
  \footnotesize
  \begin{tabular}{@{~}l@{~~}r@{~~}r@{~~}r@{~~}r@{~~}r@{~}}
    \toprule
    \multicolumn{5}{c}{$L^2(H^1)$-norm of the error} \\
    \midrule
    & $h=1/2$
    &$h=1/4$
    &  $h=1/8$
    & $h=1/16$ &\eoc{tt} \\
    \midrule
    $\Delta t=1/8$
    & \numbf{0.94790}
    & \num{0.601728}	
    & \num{0.380077}
    & \num{0.246941}&---\\
    $\Delta t=1/32$
    & \num{0.915857}	
    & \numbf{0.61166}
    & \num{0.353779}	
    & \num{0.185639}&\numQ{0,411666558146006}\\
    $\Delta t=1/128$
    & \num{0.912044}
    & \num{0.618749}
    & \numbf{0.35496}
    & \num{0.181816}&\numQ{0,030020667327612}\\
    $\Delta t=1/512$
    & \num{0.911415}
    & \num{0.621017}
    & \num{0.355871}
    & \numbf{0.18193}&\numQ{-0,000904297174656}\\
    \midrule
    \multicolumn{1}{r}{\eoc{x}}
    &{---}
    &\numQ{0.553475352551446}
    &\numQ{0.803278389846293}
    &\numQ{0.967970911145681}\\
    \midrule
    \multicolumn{1}{r}{\underline{\eoc{xtt}}}
    &{---}
    &\numQ{0.632004935122876}
    &\numQ{0.785073474681598}
    &\numQ{0.964272997290419}\\
    \bottomrule
  \end{tabular}
  \begin{tabular}{@{~}l@{~~}r@{~~}r@{~~}r@{~~}r@{~~}r@{~}}
    \toprule
    \multicolumn{5}{c}{$L^{\infty}(L^2)$-norm of the error} \\
    \midrule
    & $h=1/2$
    &$h=1/4$
    &  $h=1/8$
    & $h=1/16$ & \eoc{tt}\\
    \midrule
    \vphantom{$\Delta t=1/8$}
    & \numbf{0.163533}
    & \num{0.12481}
    & \num{0.143253}
    & \num{0.156605} & ---\\
    \vphantom{$\Delta t=1/32$}
    & \num{0.205187}
    & \numbf{0.110691}
    & \num{0.0347869}
    & \num{0.0378331} & \numQ{2,04940937586473}\\
    \vphantom{$\Delta t=1/128$}
    & \num{0.233605}	
    & \num{0.134622}
    & \numbf{0.0309132}
    & \num{0.00909762}& \numQ{2,05608791314427} \\
    \vphantom{$\Delta t=1/512$}
    & \num{0.241299}
    & \num{0.141716}	
    & \num{0.0351604}
    & \numbf{0.00775292} & \numQ{0,230749397447855}\\
    \midrule
    &{---}
    &\numQ{0.767819286292212}
    &\numQ{2.01097926199167}
    &\numQ{2.18113979991665}\\
\midrule
    &{---}
    &\numQ{0.563043867425114}
    &\numQ{1.84024301651728}
    &\numQ{1.99541132005739}\\
    \bottomrule
  \end{tabular}
\end{center}
\vspace*{-0.25cm}
\end{table}

\smallskip
\noindent{\bf Experiment 1.} We consider the transport--diffusion equation \eqref{transport}
on a unit sphere $\Gamma(t)$ moving with the constant velocity $\mathbf{w}=(0.2,0,0)$ for $t \in [0,T]$, $T=1$.
The level-set function $\phi$,
$$\phi=|\bx-\mathbf{c}(t)|-1,$$ with $\mathbf{c}(t)=t\mathbf{w}$
describes a sphere with radius $1$ that moves along $\bw$. We notice that $\phi$ is a signed distance function. The initial data is given by
\[
\G{0}{}:=\{\bx\in\mathbb{R}^3~:~|\bx|=1\},\quad u(\bx,0)=1+x_1+x_2+x_3.
\]
One easily checks that the exact solution is given by
$u(\bx,t)=1+(x_1+x_2+x_3-0.2t )\exp(-2t)$ and that $\xi = 0.1$ in \eqref{e:xi}. For sufficiently small $h$ we can assume that $\xi_h \approx 0.1$ (cf, \eqref{cond2}) which ensures unique solvability of every time step for $\Delta t \leq 2$.

The error measures for the backward Euler method are shown in Tables~\ref{tab:table1} and Table \ref{tab:table2} for the different scalings for $\stab$.
 In both cases we observe an $\mathcal{O}(h)$-convergence in the $L^2(H^1)$-norm. The initial temporal resolution is already so high that the spatial error is always dominating and we do not observe the linear convergence in time, yet. However, for the $L^\infty(L^2)$-norm we observe a  convergence with $h^2+\Delta t$.
The impact of the scaling of $\stab$ on the results is very small
which can be seen as some robustness of the method (in view of accuracy) with respect
to the stabilization parameter $\stab$.

\begin{table}
\caption {Maximum condition number in Experiment~1 for two different choices for $\stab$. Here, \eoca{x} and \eoca{t} refer to the coarsest time level, $\Delta t = 1/2$ and the coarsest space level $h=1/2$, respectively.}\label{tab:condnumber}
\vspace{-0.2cm}
  \footnotesize
\begin{center}
\begin{tabular}{@{~}l@{~~}r@{~~}r@{~~}r@{~~}r@{~~}r@{~}}
\toprule
\multicolumn{5}{c}{$\kappa(\bA)$ for $\stab=4$}\\
\midrule
&$h=1/2$
&$h=1/4$
&$h=1/8$
&$h=1/16$&\eoca{t}\\
\midrule
$\Delta t= 1/2$
&\numit{5.710E+01}
&\num{1.109E+02}
&\num{3.320E+02}
&\num{1.791E+03}& ---\\
$\Delta t=1/4 $
&\num{6.178E+01}
&\numbf{8.999E+01}
&\num{1.815E+02}
&\num{6.262E+02}&\numQ{-0.113649125377433}\\
$\Delta t=1/8$
&\num{7.199E+01}
&\numit{7.66E+01}
&\numbf{1.617E+02}
&\num{3.364E+02}&\numQ{-0.220656647361993}\\
$\Delta t=1/16$
&\num{1.059E+02}
&\num{8.290E+01}
&\num{1.577E+02}
&\numbf{3.015E+02}&\numQ{-0.556834165890068}\\
$\Delta t=1/32$
&\num{1.875E+02}
&\num{1.004E+02}
&\numit{1.43E+02}
&\num{2.940E+02}&\numQ{-0.824188006278269}\\
$\Delta t=1/64$
&\num{3.549E+02}
&\num{1.524E+02}
&\num{1.540E+02}
&\num{2.923E+02}&\numQ{-0.920521978790339}\\
$\Delta t=1/128$
&\num{6.895E+02}
&\num{2.770E+02}
&\num{1.793E+02}
&\numit{2.91E+02}&\numQ{-0.958137977340398}\\
  \midrule
    \multicolumn{1}{r}{\eoca{x}}
  &{---}
    &\numQ{-0,95769671479887}
    &\numQ{-1,58192387607258}
    &\numQ{-2,43151019063904}\\
\midrule
    \multicolumn{1}{r}{\underline{\eoc{xt}}}
  &{---}
    &\numQ{-0.656273947499022}
    &\numQ{-0.845483080611797}
    &\numQ{-0.898838323313786}\\
\midrule
    \multicolumn{1}{r}{\underline{\underline{\eoc{xtt}}}}
  &{---}
    &\numQ{-0.423853646562114}
    &\numQ{-0.900598849740795}
    &\numQ{-1.02500400612989}\\
\bottomrule
\end{tabular}
\begin{tabular}{@{~}l@{~~}r@{~~}r@{~~}r@{~~}r@{~}r@{~}}
\toprule
  \multicolumn{5}{c}{$\kappa(\bA)$ for $\stab = \norm{\bw}_{\infty} \!\!+ \nu(\delta_h+h)^{-1}$}\\
\midrule
&$h=1/2$
&$h=1/4$
&$h=1/8$
&$h=1/16$&\eoca{t}\\
\midrule
&\numit{1.78E+02}
&\num{2.389E+02}
&\num{4.797E+02}
&\num{2.070E+03}&---\\
&\num{1.937E+02}
&\numbf{1.692E+02}
&\num{2.313E+02}
&\num{7.430E+02}&\numQ{-0.121946712749494}\\
&\num{2.313E+02}
&\numit{1.66E+02}
&\numbf{1.792E+02}
&\num{3.306E+02}&\numQ{-0.255941312032937}\\
&\num{3.575E+02}
&\num{1.761E+02}
&\num{1.686E+02}
&\numbf{2.050E+02}&\numQ{-0.628177975916924}\\
&\num{6.698E+02}
&\num{2.165E+02}
&\numit{1.67E+02}
&\num{1.716E+02}&\numQ{-0.905787134153602}\\
&\num{1.301E+03}
&\num{3.347E+02}
&\num{1.686E+02}
&\num{1.649E+02}&\numQ{-0.957818680898442}\\
&\num{2.562E+03}
&\num{6.142E+02}
&\num{1.959E+02}
  &\numit{1.71E+02}&\numQ{-0.977649513623851}\\
\midrule
    \multicolumn{1}{r}{\vphantom{\eoc{x}}}
    &{---}
    &\numQ{-0.42452961260257}
    &\numQ{-1.00572558574516}
    &\numQ{-2.10942642307256}\\
\midrule
    \multicolumn{1}{r}{\vphantom{\underline{\eoc{xt}}}}
    &{---}
    &\numQ{0.007314767273381}
    &\numQ{-0.082841068937654}
    &\numQ{-0.194053272335204}\\
\midrule
    \multicolumn{1}{r}{\vphantom{\underline{\underline{\eoc{xtt}}}}}
    &{---}
    &\numQ{0.100693999619473}
    &\numQ{-0.0008664861127128}
    &\numQ{-0.034148222411846}\\
\bottomrule
\end{tabular}
\end{center}
\vspace*{-0.25cm}
\end{table}
The maximal condition numbers for each simulation are shown in Table~\ref{tab:condnumber} for the different scalings. We first discuss $\stab = 4$.
For fixed $\Delta t=1/2$, we observe that the condition number increases like $\mathcal{O}(h^{-2})$ which is slightly better than predicted.
For fixed $h=1/2$, the condition number increases with order $\mathcal{O}(\Delta t^{-1})$.
When we refine $h$ and $\Delta t $ simultaneously, we observe the predicted $\mathcal{O}(h^{-1})$ behavior
for both cases with $\Delta t\sim h$ and $\Delta t\sim h^2$.
For the scaling $\stab = \norm{\bw}_{\infty} + \frac{\nu}{\delta_h + h}$ we observe slightly higher condition numbers, the same behavior for fixed $h$ or fixed $\Delta t$, but a better scaling for  $\Delta t\sim h$ and $\Delta t\sim h^2$. For $\Delta t\sim h$ the condition number only grows slowly with $h^{-1}$ (not even linear as predicted) and is constant for $\Delta t\sim h^2$.

Finally, we do experiments for the BDF2 scheme, cf. Remark \ref{rem:bdf2}. In this case, we expect that the method is of $\mathcal{O}(\Delta t^2)$ accuracy.
This is clearly shown in Table~\ref{tab:convergencerateH1} when we refine both $h$ and $\Delta t$ (with $\Delta t\sim h$).
In these tests, we only considered $\stab=4$.
We notice that the system matrix is different from that of the backward Euler scheme only by a different coefficient in front of the mass matrix.
Therefore, the algebraic stability of the BDF2 scheme is the same as that of the backward
Euler scheme, and is covered by the analysis in Section 6.
These results indicate that the stabilized TraceFEM
method can be generalized to higher order time discretization schemes.
\begin{table}\small
\caption{$L^2(H^1)$- and $L^{\infty}(L^2)$-norm error in Experiment~1 with BDF2 scheme and $\stab = 4$.}\label{tab:convergencerateH1}
\vspace{-0.2cm}
\begin{center}
  \footnotesize
\begin{tabular}{@{~}l@{~~}r@{~~}r@{~~}r@{~~}r@{~~}r@{~}}
\toprule
  \multicolumn{5}{c}{$L^2(H^1)$-norm of the error} \\
\midrule
&$h=1/2$
&$h=1/4$
&$h=1/8$
&$h=1/16$&\eoc{t}\\
\midrule
$\Delta t=1/8$
&\numbf{1.03759}
&\num{0.684156}
&\num{0.372373}
&\num{0.188597}
& ---
  \\
$\Delta t=1/16$
&\num{0.977385}
&\numbf{0.66551}
&\num{0.360772}
&\num{0.182046}
& \numQ{0.05100368578163}
  \\
$\Delta t=1/32$
&\num{0.947208}
&\num{0.654957}
&\numbf{0.35559}
&\num{0.179307}
& \numQ{0.021871230608172}
\\
$\Delta t=1/64$
&\num{0.932367}
&\num{0.649682}
&\num{0.353201}
&\numbf{0.17807}
& \numQ{0.009987329305863}
  \\
  \midrule
    \multicolumn{1}{r}{\eoc{x}}
  &{---}
    &\numQ{0.521164208299921}
    &\numQ{0.879244307006817}
    &\numQ{0.988044945975833}\\
  \midrule
    \multicolumn{1}{r}{\underline{\eoc{xt}}}
    &{---}
    &\numQ{0.640704231031356}
    &\numQ{0.904245591466511}
    &\numQ{0.997770271225224}\\
\bottomrule
\end{tabular}
\begin{tabular}{@{~}l@{~~}r@{~~}r@{~~}r@{~~}r@{~~}r@{~}}
\toprule
  \multicolumn{5}{c}{$L^{\infty}(L^2)$-norm of the error} \\
\midrule
&$h=1/2$
&$h=1/4$
&$h=1/8$
&$h=1/16$&\eoc{t}\\
\midrule
\vphantom{$\Delta t=1/8$}
&\numbf{0.394141}
&\num{0.183495}
&\num{0.0492045}
&\num{0.0342043}
& ---
\\
\vphantom{$\Delta t=1/16$}
&\num{0.382957}
&\numbf{0.173536}
&\num{0.0371924}
&\num{0.0143018}
&\numQ{1,25798097163806}
\\
\vphantom{$\Delta t=1/32$}
&\num{0.377829}
&\num{0.169889}
&\numbf{0.0350717}
&\num{0.00894927}
&\numQ{0.676354823195156}
\\
\vphantom{$\Delta t=1/64$}
&\num{0.376393}
&\num{0.169011}
&\num{0.0348942}
&\numbf{0.00825413}
&\numQ{0.11665384471925}
 \\
\midrule
    &{---}
    &\numQ{1.15512265099631}
    &\numQ{2.27605798526386}
    &\numQ{2.07979919066443}\\
\midrule
    &{---}
    &\numQ{1.18347685116084}
    &\numQ{2.30685571267864}
    &\numQ{2.08711929714897}\\
\bottomrule
\end{tabular}
\end{center}
\vspace*{-0.25cm}
\end{table}

\smallskip
\noindent{\bf Experiment~2.}
The setup of this experiment is similar to the previous one.  The transport velocity is given by a standing vortex, $\mathbf{w}=(-0.2\pi x_2, 0.2\pi x_1,0)$ for $t\in[0,T]$, $T=1$.
Initially, the sphere with radius $1$ is located off the center.
The initial data is
\[
\G{0}{}:=\{\bx\in\mathbb{R}^3~:~|\bx-\bx_0|=1\},\quad u|_{t=0}=1+(x_1-0.5)+x_2+x_3,
\]
with $\bx_0=(0.5,0,0)$.
As the level-set function we choose
$$
\phi=(x_1- 0.5\cos 0.2\pi t)^2+(x_2- 0.5\sin 0.2\pi t)^2+x_3^2-1.
$$
which is not a signed distance function.
 Now $\bw$ revolves the sphere around the center of the domain without changing its shape.
One checks that the exact solution to \eqref{transport} is given by
\[
u(\bx,t)=(x_1(\cos(0.2\pi t)-\sin(0.2\pi t))+x_2(\cos(0.2\pi t)+\sin(0.2\pi t))+x_3+0.5 )\exp(-2t).
\]
and that there hold the bounds $\xi \leq 0.6$ and $\|\wn\|_{\infty} \leq \frac{\pi}{10}$. Hence, for $h$ sufficiently small $\Delta t \leq 0.4$ ensures unique solvability in every time step.

{
\begin{table}
\caption{$L^2(H^1)$- and $L^{\infty}(L^2)$-norm error in Experiment~2 with backward Euler and $\stab = \norm{\bw}_{\infty}+\nu{(\delta_h + h)}^{-1}$.}\label{tab:ErrorH1_Ex2BE}
\vspace{-0.2cm}
\begin{center}
  \footnotesize
\begin{tabular}{@{~}l@{~~}r@{~~}r@{~~}r@{~~}r@{~~}r@{~}}
\toprule
  \multicolumn{5}{c}{$L^2(H^1)$-norm of the error} \\
\midrule
&$h=1/2$
&$h=1/4$
&$h=1/8$
&$h=1/16$&\eoc{tt}\\
\midrule
$\Delta t=1/8$
&\numbf{1.00254}
&\num{0.650837}
&\num{0.382428}
&\num{0.337564}
&---
\\
$\Delta t=1/32$
&\num{0.99045}
&\numbf{0.649988}
&\num{0.348426}
&\num{0.183455}
&\numQ{0.879734828697388}
\\
$\Delta t=1/128$
&\num{0.994787}
&\num{0.656637}
&\numbf{0.349036}
&\num{0.178477}
&\numQ{0.039688056458053}
\\
$\Delta t=1/512$
&\num{0.996284}
&\num{0.659073}
&\num{0.349911}
&\numbf{0.178506}
&\numQ{-0.000234398591965}
\\
  \midrule
    \multicolumn{1}{r}{\eoc{x}}
    &{---}
    &\numQ{0.596118785191955}
    &\numQ{0.913450250670887}
    &\numQ{0.97101545127719}\\
\midrule
    \multicolumn{1}{r}{\underline{\eoc{xtt}}}
    &{---}
    &\numQ{0.62517481078584}
    &\numQ{0.897037238084716}
    &\numQ{0.967403278050209}\\
\bottomrule
\end{tabular}
\begin{tabular}{@{~}l@{~~}r@{~~}r@{~~}r@{~~}r@{~~}r@{~}}
\toprule
\multicolumn{5}{c}{$L^{\infty}(L^2)$-norm of the error}\\
\midrule
&$h=1/2$
&$h=1/4$
&$h=1/8$
&$h=1/16$
& \eoc{tt}
  \\
\midrule
\vphantom{$\Delta t=1/8$}
&\numbf{0.301253}
&\num{0.169439}
&\num{0.170269}
&\num{0.229105}
&---
\\
\vphantom{$\Delta t=1/32$}
&\num{0.335412}
&\numbf{0.119227}
&\num{0.0496647}
&\num{0.0454392}
&\numQ{2.333999604}
\\
\vphantom{$\Delta t=1/128$}
&\num{0.361231}
&\num{0.131671}
&\numbf{0.0345041}
&\num{0.0132899}
&\numQ{1.773607186}
\\
\vphantom{$\Delta t=1/512$}
&\num{0.36829}
&\num{0.145938}
&\num{0.036698}
&\numbf{0.00908549}
&\numQ{0.54869402}
\\
\midrule
    &{---}
    &\numQ{1.335486637}
    &\numQ{1.99158}
    &\numQ{2.014065}\\
\midrule
    &{---}
    &\numQ{1.337264625}
    &\numQ{1.788871276}
    &\numQ{1.925131573}\\
\bottomrule
\end{tabular}
\end{center}
\vspace*{-0.25cm}
\end{table}
}
The numerical results are similar to those in Experiment 1. For simplicity, we show
only the errors for the backward Euler scheme with $\stab=\norm{\bw}_{\infty}+\frac{\nu}{\delta_h + h}$ in Table~\ref{tab:ErrorH1_Ex2BE}. If one refines both $\Delta t$ and $h$ with constraint $\Delta t\sim h^2$, the first order of convergence in the surface $L^2(H^1)$-norm and the second order in the surface $L^\infty(L^2)$-norm with respect to $h$ are again observed.
This example demonstrates that the numerical method works well even if the level-set function is not a signed distance function.

\smallskip
\noindent{\bf Experiment~3.} In this experiment, we consider a shrinking sphere and solve \eqref{transport} with a source term
on the right-hand side.
The bulk velocity field is given by
$
\mathbf{w}={ -{\frac34 e^{-t/2}}}\mathbf{n},
$
for $t\in[0,T]$, $T=0.5$.
Here $\mathbf{n}$ is the unit outward normal on $\Gamma(t)$. $\G{0}{}$ is the  sphere with radius $r_0=1.5$.
The level-set function is chosen as a signed distance function $\phi=|\bx|-r(t)$, with $r(t)=r_0e^{-t/2}$.
One computes $\xi = -1$ and $\|\wn\|_\infty = \frac34$ and with
the  right-hand side
$
f(\bx,t)=(-1.5 e^{t}+\frac{16}{3}e^{2t})x_1x_2x_3.
$
the exact solution  $u(\bx,t)=(1+x_1x_2x_3) e^{t}$.%
\begin{table}
\caption{$L^2(H^1)$- and $L^{\infty}(L^2)$-norm error in Experiment~3 with backward Euler and $\stab =   \norm{\bw}_{\infty}+\nu{(\delta_h + h)}^{-1}$.}\label{tab:ErrorH1_Ex3BE}
\vspace{-0.2cm}
\begin{center}
  \footnotesize
\begin{tabular}{@{~}l@{~~}r@{~~}r@{~~}r@{~~}r@{~~}r@{~}}
\toprule
  \multicolumn{5}{c}{$L^2(H^1)$-norm of the error} \\
\midrule
&\hspace{-0.5cm}$h=1/2$
&$h=1/4$
&$h=1/8$
&$h=1/16$
&\eoc{tt}  \\
\midrule
$\Delta t=1/8$
&\numbf{1.22087}
&\num{0.725252}
&\num{0.408168}
  &\num{0.258002}
  & ---
  \\
$\Delta t=1/32$
&\num{1.13401}
&\numbf{0.670837}
&\num{0.35655}
  &\num{0.189181}
  &\numQ{0.447615047}
  \\
$\Delta t=1/128$
&\num{1.11798}
&\num{0.660794}
&\numbf{0.34914}
  &\num{0.179906}
  &\numQ{0.072523899525935}
\\
$\Delta t=1/512$
&\num{1.11517}
&\num{0.658839}
&\num{0.348084}
&\numbf{0.179016}
  &\numQ{0.007154764645359}
\\
\midrule
    \multicolumn{1}{r}{\eoc{x}}
    &{---}
    &\numQ{0.75926579}
    &\numQ{0.92049046}
    &\numQ{0.95934696}\\
\midrule
    \multicolumn{1}{r}{\underline{\eoc{xtt}}}
    &{---}
    &\numQ{0.86387543}
    &\numQ{0.94215661}
    &\numQ{0.96371714}\\
\bottomrule
\end{tabular}
\begin{tabular}{@{~}l@{~~}r@{~~}r@{~~}r@{~~}r@{~~}r@{~}}
\toprule
\multicolumn{5}{c}{$L^{\infty}(L^2)$-norm of the error}\\
\midrule
&$h=1/2$
&$h=1/4$
&$h=1/8$
  &$h=1/16$
&\eoc{tt}  \\
\midrule
\vphantom{$\Delta t=1/8$}
&\numbf{0.569191}
&\num{0.319445}
&\num{0.249476}
&\num{0.23927}
&---
\\
\vphantom{$\Delta t=1/32$}
&\num{0.530595}
&\numbf{0.225269}
&\num{0.0853798}
&\num{0.070355}
&\numQ{1.765914659}
\\
\vphantom{$\Delta t=1/128$}
&\num{0.531947}
&\num{0.207719}
&\numbf{0.0635117}
&\num{0.0221062}
&\numQ{1.670201906}
\\
\vphantom{$\Delta t=1/512$}
&\num{0.53292}
&\num{0.207172}
&\num{0.0612686}
&\numbf{0.0169603}
&\numQ{0.382289362}
\\
\midrule
    &{---}
    &\numQ{1.3630899}
    &\numQ{1.75760924}
    &\numQ{1.852986197}\\
\midrule
    &{---}
    &\numQ{1.3372624057}
    &\numQ{1.826554501}
    &\numQ{1.904860698}\\
\bottomrule
\end{tabular}
\end{center}
\vspace*{-0.25cm}
\end{table}
Table~\ref{tab:ErrorH1_Ex3BE} shows the error norms for various time steps $\Delta t$ and mesh sizes $h$. The results are consistent with the previous experiments and our analysis.

\smallskip
\noindent{\bf Experiment 4.}
Additionally, we consider a problem where two initially separated spheres merge to one surface. The numerical results are similar to that by the method based on a fast marching extension in \cite{olshanskii2017trace}, i.e. a stable numerical solution. This indicates that the proposed method is robust also problems with topology changes (which is not covered by our numerical analysis).

\section{Conclusions and open problem} \label{s:Conclusions}
In this paper we introduced a new numerical method for PDEs for evolving surfaces using the example of a scalar transport diffusion equation.
The main feature of the method is its simplicity.
With the help of the stabilization which also provides a meaningful extension, standard time integration methods based on finite differences can be applied and combined with a TraceFEM for the spatial discretization. The two components, time and space discretizations can be exchanged so that higher order in space and/or in time methods can be used, if desired (and available).
Besides the introduction of the method, we carried out a careful a priori error analysis yielding optimal order estimates and reasonable condition number bounds. For the accessibility of the paper we made several restrictions and simplifications. We mention aspects where we think that an extension of our results beyond these restrictions is worth pursuing.\smallskip


The geometry in the analysis part of the paper is always described by a level set function which has the signed distance property. 
We made this assumption as it simplified the - still technical enough - analysis. However, we believe that this assumption could be replaced with the much milder assumption $c\leq\norm{\nabla \phi}\leq c^{-1}$ for some $0 < c < 1$ in the vicinity of the surface.\smallskip

The exponential growth in the a priori error analysis is due to the divergence term in \eqref{transport} which is not sign definite. For a non-negative divergence or strong diffusion the exponential growth vanishes which can be used for improved stability and error bounds, cf. Remark \ref{rem:expgrowth}.  \smallskip

Often practically relevant transport--diffusion equations are transport dominated. In these cases additional convection stabilizations may be desired. For stationary surfaces this can be dealt with a streamline--diffusion--type stabilization for TraceFEM as in \cite{olshanskii2014stabilized} or a discontinuous Galerkin TraceFEM discretization as in \cite{burman2016cut}. These techniques can be combined with our time marching method.\smallskip

The analysis in this paper only treats the backward Euler time discretization method although the methodology allows for a larger class of time stepping schemes. In Remark \ref{rem:bdf2} we also commented on adaptations of the analysis for a BDF2 scheme. The application and analysis of Crank--Nicolson or Runge--Kutta type schemes for this discretization has not been considered yet, but is an interesting natural extension of the method.\smallskip

The a priori error results presented in section \ref{sec:aprioriest} give bounds for the error at fixed times and an $L^2(H^1)$-type bound in space--time using energy-type arguments. We expect that the application of duality techniques can improve these bounds yielding an additional order in space in weaker norms such as $L^\infty(L^2)$-type space--time norms.\smallskip

The method and its analysis allow for higher order discretizations in space. However, the realisation of geometrically high order accurate discretizations is a non-trivial task, cf. Remark \ref{rem:integration}. A combination of recent developments in the accurate numerical integration on level set domains with this time discretization approach is an interesting topic for future research. \smallskip

Finally, an analog of the presented approach for PDEs posed in time-depended volumetric domains or volumetric domains with evolving interfaces was recently studied in~\cite{lehrenfeld2018eulerian}. For volumetric domains, the method is based on  new implicit extensions  of  finite element functions for geometrically unfitted domains. It naturally combines with the present method for bulk--surface coupled systems.
\vspace*{-0.2cm}

\ifarxiv
\appendix
\section*{Appendix}
In this section we collect two auxiliary lemmas and the proof of Lemma \ref{lemcrucial}.
\begin{lemma} \label{lem:app:wd}
  Functions in $\mathcal{V}(t)$, cf. \eqref{eq:defV}, have a weak derivative in $L^2(\O(\G{}{}(t)))$, $t \in (0,T]$.
\end{lemma}
\begin{proof}
    Let $\{u_k\}_{k}$ be a Cauchy sequence in $\mathcal{V}(t)$, $u_k \in \{ v \in {C^2}(\mathcal{O}(\Gamma(t)) \mid ~ \nabla v \cdot \nabla \phi = 0 \}$ with $u_k \to u$ and let $n=\nabla \phi$, $i\in\{1,..,d\}$ and $p(x)$ the closest point projection. We show that $\partial_{x_i} u_k$ is also Cauchy:
 \begin{align*}
    \Vert & \partial_{x_i} (u_k - u_l) \Vert_{\mathcal{O}(\Gamma(t))}^2
  = \int_{\mathcal{O}(\Gamma(t))} (\partial_{x_i} (u_k - u_l))^2 ~ dx
     = \int_{\mathcal{O}(\Gamma(t))} (\partial_{x_i} (u_k - u_l))^2(p(x)) ~ dx \\
    & +  \int_{\mathcal{O}(\Gamma(t))} \int_{0}^{\phi(x)} \partial_n((\partial_{x_i} (u_k - u_l))^2(p(x) + s \cdot \nabla \phi)) ~ ds ~ dx \\
   \lesssim & \underbrace{\Vert \nabla_\Gamma (u_k - u_l) \Vert_{\Gamma(t)}^2}_{\stackrel{k,l \to \infty}{\longrightarrow} 0} + \underbrace{\Vert \partial_n (u_k - u_l) \Vert_{\Gamma(t)}^2}_{=0} \\
    & +  \int_{\mathcal{O}(\Gamma(t))} \int_{0}^{\phi(x)} 2 (\partial_{x_i} (u_k - u_l))(\partial_{x_i} \underbrace{\partial_n (u_k - u_l)}_{=0})(p(x) + s \cdot \nabla \phi) ~ ds ~ dx
 \end{align*}
Hence, $\nabla u_k \stackrel{k \to \infty}{\longrightarrow} g \in L^2(\mathcal{O}(\Gamma(t)))$ and $g$ is the weak derivative to $u$.
\end{proof}
The following lemma states a result that we make use of in the proof of Lemma \ref{lem2a}.
\begin{lemma} \label{lem:surfmeasratio}
  Let $\mu_h^k$ be the ratio of surface measure between $\G{k}{h}$, $k = n-1,n$ and $\G{n-1}{}$ so that with $\bx^k = \bl{n-1,k}(\bx),~\bx\in\G{n-1}{}$, there holds, cf. \eqref{aux2a},
  \begin{equation*}
    \mu^k_h(\bx^k)d\bs^k_h(\bx^k)= d\bs^{n-1}(\bx),\quad \bx\in\G{n-1}{} \quad \text{ and } \quad  \mu^k_h(\bx^k)= a_1^k ~ a_2^k ~ b^k
  \end{equation*}
  \begin{equation*}
    \text{ with }
  a_1^k:= (1-\phi^{n-1}(\bx^k)\kappa_1(\bx^k)), \quad
a_2^k:= (1-\phi^{n-1}(\bx^k)\kappa_2(\bx^k)), \quad
b^k := \bn_h^k(\bx^k)\cdot \bn^{n-1}(\bx^k),
\end{equation*}
where we recall $\bn_h^k=\nabla\phi_h^k/|\nabla\phi_h^k|$ and set $ \bn^{n-1} := \nabla\phi^{n-1}$.
Further let $h$ be such that \eqref{cond4} is fulfilled.
Then, there holds
\begin{equation}
  |1-\mu^n_h(\bx^n)/\mu^{n-1}_h(\bx^{n-1})| \lesssim
  c_{\ref{lem:surfmeasratio}}
  \Delta t.
\end{equation}
for some
$
  c_{\ref{lem:surfmeasratio}}
$
independent of $h$, $\Delta t$ and $n$.
\end{lemma}
\begin{proof}
We split this term into several parts:
\begin{align}
  |1-\mu^n_h(\bx^n)/\mu^{n-1}_h(\bx^{n-1})|
  & = \frac{1}{|\mu_h^{n-1}(\bx^{n-1})|} | \mu_h^{n}(\bx^{n}) - \mu_h^{n-1}(\bx^{n-1}) |
    \leq c | a_1^{n}a_2^{n}b^{n} - a_1^{n-1}a_2^{n-1}b^{n-1} | \nonumber\\
  & \leq c | (a_1^{n}-a_1^{n-1})a_2^{n}b^{n} + a_1^{n-1}(a_2^{n}-a_2^{n-1})b^{n} + a_1^{n-1}a_2^{n-1}(b^n - b^{n-1}) | \nonumber \\
  & \leq c \left(|a_1^{n}-a_1^{n-1}| + |a_2^{n}-a_2^{n-1}| + |b^n - b^{n-1}|\right), \label{eq:split}
\end{align}
where we exploited that there is a constant $c>1$ independent of $h$, $\Delta t$ and $n$ so that $c^{-1} \leq a_i^k, b^k \leq c,~i=1,2,~k=n-1,n$.
We start with bounds for $|a_i^n - a_i^{n-1} |$.
Since $\phi^{n-1}(\bx)$ is the signed distance function for $\G{n-1}{}$, and $\blk{n-1}{k}$ is the lift operator along the normal directions to $\G{n-1}{}$, it holds
\[
|\phi^{n-1}(\bx^{n-1})-\phi^{n-1}(\bx^n)|=|\bx^{n-1}-\bx^n|.
\]
To estimate the distance on the right-hand side, we note that
\[
(\bx^{n-1}-\bx^n)\cdot\nabla\phi_h^{n-1}(\by) \leq \vert \phi_h^{n-1}(\bx^{n-1})-\phi_h^{n-1}(\bx^n) \vert = \vert \phi_h^{n-1}(\bx^n) \vert \text{ for some } \by \in \operatorname{conv}(\bx^{n-1},\bx^n).
\]
For the same $\by$ we have $(\bx^{n-1}-\bx^n) \parallel \nabla \phi^{n-1}(\by)$ which together with $|\nabla\phi^{n-1}|=1$ and \eqref{phi_h} yields
\begin{equation}\label{aux7}
  \begin{split}
  |\bx^{n-1}-\bx^n|
  & = |( \bx^{n-1}-\bx^n) \cdot\nabla \phi^{n-1}(\by)| \\
  & \leq |( \bx^{n-1}-\bx^n)\cdot \nabla \phi_h^{n-1}(\by)|
  + |( \bx^{n-1}-\bx^n) \cdot\nabla \left(\phi^{n-1}(\by)-\phi_h^{n-1}(\by)\right)| \\
  & \leq |\phi_h^{n-1}(\bx^n)| + c h^q |\bx^{n-1}-\bx^n|
    = |\phi_h^{n-1}(\bx^n) - \phi_h^{n}(\bx^n)| + c h^q |\bx^{n-1}-\bx^n|  \\
    & \leq c \delta_n + c h^q |\bx^{n-1}-\bx^n| \quad \Longrightarrow \quad
  |\bx^{n-1}-\bx^n| \lesssim \frac{\delta_n}{1-ch^q},
  \end{split}
\end{equation}
where for the last estimate we made use of \eqref{phi_hba} and \eqref{e:delta}.
We note that \eqref{aux7} estimates the `distance' between $\G{n}{h}$ and $\G{n-1}{h}$ measured in the normal directions to $\G{n-1}{}$.
From \eqref{eq:kappa} we find
$$
|(\bx^{n}-\bx^{n-1})\cdot\nabla \kappa_i(\by)|
=|\bx^{n}-\bx^{n-1}||\bn^{n-1}\cdot\nabla \kappa_i(\by)|
= \kappa_i^2(\by) |\bx^{n}-\bx^{n-1}|
$$
for any $\by \in \operatorname{conv}(\bx^{n-1},\bx^n)$. We use this to estimate,
\begin{align}\label{aux5}
  |a_i^{n} - a_i^{n-1}|
  & = |\phi^{n-1}(\bx^n)\kappa_i(\bx^n) - \phi^{n-1}(\bx^{n-1})\kappa_i(\bx^{n-1})|  \nonumber \\
  & = | \nabla ( \phi^{n-1}(\by)\kappa_i(\by)) \cdot (\bx^n - \bx^{n-1}) | \qquad \text{ for some } \by \in \operatorname{conv}(\bx^{n-1},\bx^n), \nonumber  \\
  & \leq |\nabla \phi^{n-1}(\by)| |\kappa_i(\by)| |\bx^{n}-\bx^{n-1}| + |\phi^{n-1}(\by)| |(\bx^{n}-\bx^{n-1})\cdot\nabla \kappa_i(\by)|\nonumber \\
  & \leq
    |\kappa_i(\by)| |\bx^{n}-\bx^{n-1}|
    +
    \underbrace{|\phi^{n-1}(\by)| |\kappa_i(\by)|}_{\lesssim \delta_n \kappa(\by)\lesssim 1} |\kappa_i(\by)| |\bx^{n}-\bx^{n-1}|
  \lesssim c_{\delta} \norm{\kappa}_{\infty,I_n} \norm{\wn}_{\infty,I_n} \Delta t.
\end{align}
From \eqref{phi_h}, \eqref{aux7} and the smoothness of $\phi$, $\|\phi\|_{C^2(\Gs)} \leq c_{\phi}$ we also conclude
\begin{align}
|b^{n-1}\!- b^n|\! & = |\bn_h^{n-1}(\bx^{n-1})\!\cdot\! \bn^{n-1}(\bx^{n-1})-\bn_h^{n}(\bx^{n})\!\cdot\! \bn^{n-1}(\bx^{n})| \nonumber\\[-4ex]&
\le
|\bn_h^{n-1}(\bx^{n-1})\!\cdot\! \bn^{n-1}(\bx^{n-1})-\bn_h^{n}(\bx^{n})\!\cdot\! \bn^{n}(\bx^n)|+|\bn_h^{n}(\bx^{n})\!\cdot\! \overbrace{(\nabla\phi^{n}(\bx^{n})-\nabla\phi^{n-1}(\bx^n))}^{= (\bn^n(\bx^n)-\bn^{n-1}(\bx^n))}|
\nonumber\\&\lesssim
|\bn_h^{n-1}(\bx^{n-1})\!\cdot\! \bn^{n-1}(\bx^{n-1}) -\bn_h^{n}(\bx^{n})\!\cdot\! \bn^{n}(\bx^{n})| + c_{\phi}  \Delta t
\nonumber\\&\lesssim
|\bn_h^{n-1}(\bx^{n-1})-\bn^{n-1}(\bx^{n-1})|^2+|\bn_h^{n}(\bx^{n})-\bn^{n}(\bx^{n})|^2+c_\phi \Delta t \lesssim h^{2q} + c_\phi \Delta t ,\label{aux6}
\end{align}
where in the last step we made use of $| \bn^k(\bx^k) | = | \bn_h^k(\bx^k) | = 1,~k=1,2$б and hence
$$
  \bn_h^{k}(\bx^k) \!\cdot\! \bn^k(\bx^k) = 1 - \frac12 | \bn_h^{k}(\bx^k)- \bn^k(\bx^k) |^2.
$$
Plugging \eqref{aux5} and \eqref{aux6} into \eqref{eq:split} and exploiting \eqref{cond4} completes the proof.
\end{proof}
\subsection*{Proof of Lemma \ref{lemcrucial}}
We prove the result in two steps.
In the first step we treat the estimates on the mapped domains for $\tilde{u} = u \circ \Phi^{-1}$ with $\bn_h^n$ replaced by $\bn^n$ resulting in the estimates \eqref{eq:step1a} and \eqref{eq:step1b} below. In the second step we incorporate the geometrical errors due to $\bn_h^n\neq \bn^n$ and transform back.

{\textbf{Step 1}}.
We proceed similar to \cite[Lemma 7.4]{grande2016analysis} and make use of the co-area formula, cf. e.g. \cite[Theorem 2.9]{DEreview}. To this end we introduce the coordinates $\bx = (\xi,s)$ so that $\bx = r_\xi(s) = \xi + s \bn^n(\xi)$ and introduce the line $R_\xi := \{r_\xi(s), s \in [-\delta_n,\delta_n]\}$.
  Then, there holds
  $$
  \int_{U_{\delta_n}} f(\bx) ~ d \bx
 = \int_{\G{n}{}} \int_{-\delta_n}^{\delta_n} J(\bx)^{-1} f(\xi + s \bn^n(\xi))  ~ ds ~ d\xi
 $$
 with the normal-Jacobian $J$ and $J(\bx)^{-1} = \operatorname{det}(I+s(\bx) H(\xi))$ where $H = D^2 s$ is the Hessian of $s(\bx)=\phi^n(\bx)$. We have with \eqref{eq:conddtkappa}\vspace*{-0.5cm}
 \begin{equation*}
   J^{-1} = \operatorname{det}(I - s(\bx) H(\bx)) \lesssim 1 + \delta_n \overbrace{\operatorname{tr}(H(\bx))}^{=\kappa(\xi)} \lesssim 1.
 \end{equation*}
Hence, $\norm{\tilde{u}}_{U_{\delta_n}(\G{n}{})}^2 \lesssim \int_{\G{n}{}} \norm{\tilde{u}(\xi,\cdot)}_{R_\xi}^2 ~ d\xi$.
For every $\xi \in \G{n}{},~s\in(-\delta_n,\delta_n)$ we have
\begin{align} \label{eq:help1}
  \tilde{u}(\xi,s)^2
  &= \tilde{u}(\xi,0)^2 + 2 \int_0^s \tilde{u}(\xi,t) \underbrace{\partial_s \tilde{u}(\xi,t)}_{= \bn^n(\xi) \cdot \nabla \tilde{u}(\xi,t)}~ d t
    \leq \tilde{u}(\xi,0)^2 + \frac{1}{\gamma}  \Vert \tilde{u} \Vert_{R_\xi}^2 + \gamma  \Vert \bn^n \cdot \nabla \tilde{u} \Vert_{R_\xi}^2.
\end{align}
Choosing $\gamma = 4 \delta_n$ and integrating over $s$ yields
$  \Vert \tilde{u} \Vert_{R_\xi}^2
  \leq  2 \delta_n  \tilde{u}(\xi,0)^2 + \frac{2 \delta_n}{4 \delta_n} \Vert \tilde{u} \Vert_{R_\xi}^2 + 8 \delta_n^2 \Vert \bn^n \cdot \nabla \tilde{u} \Vert_{R_\xi}^2$
and after integration over $\xi \in \G{n}{}$
\begin{subequations}
\begin{equation} \label{eq:step1a}
    \Vert \tilde{u} \Vert_{U_{\delta_n}(\Gamma^n)}^2
 \lesssim \delta_n \Vert \tilde{u} \Vert_{\Gamma^n}^2 + \delta_n^2   \Vert \bn^n \cdot \nabla \tilde{u} \Vert_{U_{\delta_n}(\Gamma^n)}^2.
\end{equation}
%
For $\tilde{\O}(\G{n}{h})$ we use the overlapping decomposition into
$U_{\delta_n}(\Gamma^n)$ and $\tilde{\O}_\Gamma(\G{n}{h,\pm \delta_n})$ and
make use of $\tilde{\O}_\Gamma(\G{n}{h,\pm \delta_n}) \cup U_{\delta_n}(\Gamma^n) = \tilde{\O}(\G{n}{h})$ and Lemma \ref{lem:onelayer}:
  \begin{equation}\label{aux9}
    \Vert \tilde{u} \Vert_{\tilde{\O}(\G{n}{h})}^2
     \leq
    \Vert \tilde{u} \Vert_{U_{\delta_n}(\Gamma^n)}^2
 +  \Vert \tilde{u} \Vert_{\tilde{\O}_\Gamma(\G{n}{h,\pm \delta_n})}^2
    \lesssim
    \Vert \tilde{u} \Vert_{U_{\delta_n}(\Gamma^n)}^2
 +  h \Vert \tilde{u} \Vert_{\G{n}{\pm \delta_n}}^2
      +h^2 \Vert \bn^n \cdot \nabla \tilde{u} \Vert_{\tilde{\O}(\G{n}{h})}^2.
  \end{equation}
To bound $  h \Vert \tilde{u} \Vert_{\G{n}{\pm \delta_n}}^2$ we again 
make use of \eqref{eq:help1} with $s=\pm\delta_n$, 
and choose $\gamma = 4 h$.  
This yields (after integrating over $\Gamma^n$) 
\begin{equation}\label{eq:step1b}
 h \Vert \tilde{u} \Vert_{\G{n}{\pm \delta_n}}^2
 \lesssim h  \Vert \tilde{u} \Vert_{\Gamma^n}^2 +  h^2   \Vert \bn^n \cdot \nabla \tilde{u} \Vert_{\tilde{\O}(\G{n}{h})}^2.
\end{equation}

\end{subequations}
%
{\textbf{Step 2}}.
We have to bound the normal derivative part on the mapped domain:
  \begin{align}
    \Vert \bn^n \cdot \nabla \tilde{u} \Vert_{\tilde{\O}(\G{n}{h})}^2
    & \lesssim \Vert \bn_h^n \cdot \nabla u \Vert_{\O(\G{n}{h})}^2 +   \Vert D\Phi^{-1} \cdot \bn^n \circ \Phi - \bn_h^n \Vert_{\infty,\O(\G{n}{h})}^2 \Vert \nabla u \Vert_{\O(\G{n}{h})}^2 \nonumber \\
    & \stackrel{(\ast)}{\lesssim} \Vert \bn_h^n \cdot \nabla u \Vert_{\O(\G{n}{h})}^2 +  h^{2q} \Vert \nabla u \Vert_{\O(\G{n}{h})}^2 \label{eq:subfinal}
        \lesssim \Vert \bn_h^n \cdot \nabla u \Vert_{\O(\G{n}{h})}^2 + h^{2q-2} \Vert u \Vert_{\O(\G{n}{h})}^2.
  \end{align}
  Here, we used the following estimate in $(\ast)$ for $\bx \in \O(\G{n}{h})$
  \begin{align*}
    | D\Phi^{-1} \cdot \bn^n \circ \Phi - \bn_h^n |
    & \lesssim
      \underbrace{| D\Phi^{-1} \cdot \bn^n \circ \Phi - \bn^n \circ \Phi|}_{\lesssim |D \Phi^{-1} - I| \lesssim h^q}
      + \underbrace{| \bn^n \circ \Phi - \bn^n|}_{\lesssim |\bn^n|_{W^{1,\infty}} |\Phi - \operatorname{id}| \lesssim h^{q+1}}
      + \underbrace{| \bn^n - \bn_h^n |}_{\lesssim h^q} \lesssim h^q.
  \end{align*}
  To arrive at \eqref{fund1} we combine \eqref{aux9}--\eqref{eq:step1b} with \eqref{eq:equivvol}, \eqref{eq:equivsurf} and \eqref{eq:subfinal} where the final term in \eqref{eq:subfinal} is absorbed by the left hand side due to $c\,h^{2q-2}(\delta_n + h)^2 \leq \frac12$ for $h$ and $\Delta t$ sufficiently small.
  For \eqref{fund1a} we similarly combine \eqref{eq:step1a} with \eqref{eq:equivvol}, \eqref{eq:equivsurf} and \eqref{eq:subfinal}, but also exploit \eqref{fund1}:
  \begin{align*}
\|u\|_{U_{\delta_n}(\G{n}{h})}^2
& \lesssim \delta_n \|u\|_{\G{n}{h}}^2 + \delta_n^2 \|\bn_h^n \cdot \nabla u\|_{\O(\G{n}{h})}^2  + \delta_n^2 h^{2q-2} \| u \|_{\O(\G{n}{h})}^2 & \\
& \stackrel{\eqref{fund1}}{\lesssim} \underbrace{(\delta_n + \delta_n^2 h^{2q-2} (\delta_n+h))}_{\lesssim \delta_n} \|u\|_{\G{n}{h}}^2 + \underbrace{\delta_n^2 (1+h^{2q-2} (\delta_n+h)^2)}_{\lesssim \delta_n^2}  \|\bn_h^n \cdot \nabla u\|_{\O(\G{n}{h})}^2. & \square
  \end{align*}

\fi
\bibliography{literatur}{}
\bibliographystyle{siam}
\end{document}